\documentclass[11pt,oneside,english]{amsart}
\usepackage[T1]{fontenc}
\usepackage[latin1]{inputenc}
\usepackage[a4paper]{geometry}
\usepackage{textcomp}
\usepackage{amsthm}
\usepackage{amssymb}
\usepackage{esint}

\makeatletter
\numberwithin{equation}{section}
\numberwithin{figure}{section}
\theoremstyle{plain}
\newtheorem{thm}{Theorem}[section]
  \theoremstyle{plain}
  \newtheorem{cor}[thm]{Corollary}
  \theoremstyle{plain}
  \newtheorem{lem}[thm]{Lemma}
  \theoremstyle{plain}
  \newtheorem{prop}[thm]{Proposition}
  \theoremstyle{remark}
  \newtheorem{rem}[thm]{Remark}
  \theoremstyle{plain}
  \newtheorem{conjecture}[thm]{Conjecture}

\def\makebbb#1{
    \expandafter\gdef\csname#1\endcsname{
        \ensuremath{\Bbb{#1}}}
}\makebbb{R}\makebbb{N}\makebbb{Z}\makebbb{C}\makebbb{H}\makebbb{E}\makebbb{H}\makebbb{P}\makebbb{B}\makebbb{K}\makebbb{E}

\usepackage{babel}

\usepackage{babel}

\makeatother

\usepackage{babel}

\makeatother

\usepackage{babel}

\begin{document}

\title{Moser-Trudinger type inequalities for complex Monge-Ampère operators
and Aubin's {}``hypothèse fondamentale''}

\author{Robert J. Berman, Bo Berndtsson}

\email{robertb@chalmers.se, bob@chalmers.se}

\curraddr{Mathematical Sciences - Chalmers University of Technology and University
of Gothenburg - SE-412 96 Gothenburg, Sweden }
\begin{abstract}
We prove Aubin's {}``Hypothese fondamentale'' concerning the existence
of Moser-Trudinger type inequalities on any integral compact Kähler
manifold $X.$ In the case of the anti-canonical class on a Fano manifold
the constants in the inequalities are shown to only depend on the
dimension of $X$ (but there are counterexamples to the precise value
proposed by Aubin). In the different setting of pseudoconvex domains
in complex space we also obtain a quasi-sharp version of the inequalities
and relate it to Brezis-Merle type inequalities. The inequalities
are shown to be sharp for $S^{1}-$invariant functions on the unit-ball.
We give applications to existence and blow-up of solutions to complex
Monge-Ampère equations of mean field (Liouville) type. 

\tableofcontents{}
\end{abstract}
\maketitle

\section{Introduction}

As shown by Trudinger in the seminal work \cite{tr} there is a limiting
exponential version of the critical Sobolev inequalities which, in
the case of the plane, may be formulated as the existence of positive
constants $c$ and $C$ such that

\begin{equation}
\int_{\Omega}e^{c\left(\frac{u}{\left\Vert \nabla u\right\Vert _{\Omega}}\right)^{2}}dV\leq C\label{eq:trudingers ineq}\end{equation}
for any, say smooth, function $u$ vanishing on the boundary of a
domain $\Omega$ in $\R^{2}.$ Motivated by the Nirenberg problem
for constructing conformal metrics on a real surface with prescribed
positive curvature, Moser \cite{m} obtained the sharp constant $c=4\pi$
in Trudinger's inequality \ref{eq:trudingers ineq}. The relation
to the Nirenberg problem appears in the following consequence of the
previous inequality: \begin{equation}
\log\int_{\Omega}e^{-u}dV\leq A\left\Vert \nabla u\right\Vert _{\Omega}^{2}+B\label{eq:mosers ineq}\end{equation}
Here $e^{-u}$ plays the role of the conformal factor of a metric
on $\Omega.$ As shown by Moser the inequalities also hold when the
domain $\Omega$ is replaced by the two-sphere - which is the setting
for the Nirenberg problem - and then the extremals $u$ of the inequality
correspond to metrics $g_{u}$ with constant positive curvature (with
$A=1/16\pi,$ the sharp constant). Conversely, the latter inequality
\ref{eq:mosers ineq}, with the sharp constant, implies an in equality
of the form \ref{eq:trudingers ineq}, but only with quasi-sharp constants,
i.e. the two inequalities are equivalent {}``modulo $\epsilon".$

There has been a wealth of work on extending Moser-Trudinger inequalities
in various directions in real analysis and conformal geometry (see
for example \cite{Bec,font,au2,f-m} and references therein). However,
the present paper is concerned with a different \emph{complex }variant
of these inequalities first proposed by Aubin \cite{au}, motivated
by the existence problem for Kähler-Einstein metrics with positive
curvature on complex (Fano) manifolds (see also \cite{din,d-t,p-s+}).
More precisely, we will consider two different settings: (1) compact
complex (Kähler) manifolds (without boundary) and (2) pseudoconvex
domains in $\C^{n}.$ A characteristic feature of the complex setting
(when $n>1)$ is that it is considerably more non-linear than the
real one. Indeed, the corresponding inequalities (see below) only
hold for a \emph{convex }subspace $\mathcal{H}_{0}$ of functions
$u$ and moreover the Laplacian $\Delta$ appearing in the Dirichlet
energy $\left\Vert \nabla u\right\Vert _{\Omega}^{2}(=\int_{\Omega}-u\Delta udV)$
has to be replaced by fully \emph{non-linear} complex Monge-Ampère
operators. Moreover, in the compact setting $(1)$ the space $\mathcal{H}_{0}$
is not even a cone and the corresponding Monge-Ampere operator is
not $n-$homogeneous (in contrast to the setting $(2)).$

\subsection{Statement of the main results }

\subsubsection{The setting of a compact Kähler manifold}

Let $(X,\omega)$ be a compact Kähler manifold without boundary of
complex dimension $n$ and recall that a smooth function $u$ on $X$
is called a \emph{Kähler potential} if \[
\omega_{u}:=\omega+\frac{i}{2\pi}\partial\bar{\partial}u:=\omega+dd^{c}u>0,\]
 i.e. $\omega_{u}$ is a Kähler metric in the cohomology class $[\omega]\in H^{2}(X,\R).$
We will denote by $\mathcal{H}_{0}(X,\omega)$ the convex space of
all such $u$ normalized so that $\sup_{X}u=0$ and we will consider
the following well-known functional on $\mathcal{H}_{0}(X,\omega):$
\begin{equation}
\mathcal{E}_{\omega}(u):=\frac{1}{(n+1)!}\sum_{j=0}^{n}\int_{X}u(\omega_{u})^{j}\wedge(\omega)^{n-j}\label{eq:energy intro}\end{equation}
that we will refer to as (minus) the\emph{ Monge-Ampère energy}. 
\begin{thm}
\label{thm:aubin hyp intro}Let $(X,\omega)$ be a Kähler manifold
such that $[\omega]\in H^{2}(X,\Z)\otimes\R.$ Then the following
Moser-Trudinger type inequality holds for any function $u$ in $\mathcal{H}_{0}(X,\omega)$
and positive number $k:$ \begin{equation}
\log\int_{X}e^{-ku}dV\leq Ak^{n+1}(-\mathcal{E}_{\omega}(u))+B\label{eq:m-t in theorem aubin intro}\end{equation}
 for some positive constants $A$ and $B$ (given a volume form $dV).$
More precisely, the constant $A$ may be replaced by $(1+C_{1}/k)$
and $B$ by $(1+C_{2}/k)$ for certain invariants $C_{1}$ and $C_{2}$
of $\omega$ (see \ref{eq:upper lower bounds on bergman}). 
\end{thm}
The first part of the theorem establishes a conjecture of Aubin (called
{}``Hypothèse fondamentale'' in \cite{au}) under the assumption
that the class $[\omega]$ be integral. The inequalities \ref{thm:aubin hyp intro}
are equivalent to the existence of positive constants $c$ and $C$
such that \begin{equation}
\int_{X}e^{c\left(\frac{-u}{(-\mathcal{E}(u))^{1/(n+1)}}\right)^{(n+1)/n}}dV\leq C,\label{eq:trudingers ineq on kahler}\end{equation}
providing a variant of Trudinger's inequality \ref{eq:trudingers ineq}
in the Kähler setting. It appears to be new even in the case of two-dimensional
projective space. In particular we deduce the following Sobolev type
inequalities of independent interest: for any $u$ in $\mathcal{H}_{0}(X,\omega)$
\begin{equation}
\left\Vert u\right\Vert _{L^{p}(X)}^{n+1}\leq Cp^{n}(-\mathcal{E}_{\omega}(u))\label{eq:sob intro}\end{equation}
for all $p$ in $]1,\infty[,$ for some constant $C$ only depending
on $\omega.$

The starting point of the proof of the previous theorem is the basic
fact that, in the integral case when $[\omega]\in H^{2}(X,\Z),$ the
space $k\mathcal{H}_{0}(X,\omega)$ may be identified (mod $\R)$
with the space $\mathcal{H}(kL)$ of all positively curved metrics
on the $k$th tensor of an ample line bundle $L\rightarrow X$ with
Chern class $c_{1}(L)=[\omega].$ The proof then exploits convexity
properties along geodesics of certain functionals on the space $\mathcal{H}(L)$
equipped with the Mabuchi metric (see section \ref{sub:Out-line-of-the}
for an outline of the proof).

As pointed out above Aubin's main motivation for his conjecture came
from the existence problem for positively curved Kähler-Einstein metrics
on a Fano manifold where the Kähler class $[\omega]$ is the integral
class $c_{1}(-K_{X}),$ i.e. the first Chern class of the anti-canonical
line bundle $-K_{X}$ of $X.$ In this setting, which we will refer
to as the Fano setting, he also conjectured an explicit optimal value
for $A$ which only depends on the dimension $n$ of the Fano manifold.
However, as explained in section \ref{sec:Remarks-on-the} there is
a simple counter-example to the explicit value proposed by Aubin.
Still, combining our arguments with previous work on finiteness properties
of Fano manifolds \cite{t-y,kmm,ca} we deduce the following partial
confirmation of Aubin's latter conjecture:
\begin{thm}
\label{thm:uniform fano intro}When $X$ is an $n-$dimensional Fano
manifold and $[\omega]$ is the anti-canonical class the constant
$A$ can be taken to only depend on $n$ (if $B$ is allowed to depend
on $k).$
\end{thm}
Coming back to the general setting in Theorem \ref{thm:aubin hyp intro}
we point out that the Moser-Trudinger type inequalities there will
be shown to hold as long as $\omega=c_{1}(L)$ for $L$ semi-positive
and such that the adjoint bundles $kL+K_{X}$ are base point free
for all sufficiently large positive integers $k.$ Using this and
$L^{2}-$estimates for $\bar{\partial}$ we deduce the following {}``qualitative''
Moser-Trudinger type inequality in a degenerate setting:
\begin{cor}
\label{cor:no lelong intro}Let $\omega$ be a semi-positive form
on compact complex manifold $X$ such that $[\omega]=c_{1}(L)$ with
$L$ semi-positive and big ( i.e. $\int_{X}c(L)^{n}>0).$ Then $u$
has vanishing Lelong numbers, i.e. any negative multiple of an $\omega-$psh
function $u$ with finite energy is exponentially integrable. More
precisely, if $\mathcal{E}_{\omega}(u)\geq-C$ and $\sup_{X}u=0$
then \[
\int_{X}e^{-ku}dV\leq C_{k},\]
 where the constant $C_{k}$ only depends on $C$ and $k.$
\end{cor}
The notion of finite energy is recalled in the beginning of section
\ref{sub:Vanishing-Lelong-numbers}. In the strictly positive case
the previous corollary is due to Guedj-Zeriahi \cite{g-z} and proved
using pluripotential theory. The present extension to the semi-positive
case was motivated by the work \cite{bbegz} where it is used in the
construction of Kähler-Einstein metrics on singular Fano varieties.

\subsubsection{The setting of a pseudoconvex domain in $\C^{n}$ }

Let now $\Omega$ be a pseudoconvex domain in $\C^{n}$ with smooth
boundary (for example the unitball) and set $\omega:=0.$ In this
setting we let $\mathcal{H}_{0}(\Omega)$ be the convex cone of all
smooth plurisubharmonic functions, i.e. $dd^{c}u\geq0,$ vanishing
on the boundary $\partial\Omega.$ Then the $n+1-$homogeneous functional
\begin{equation}
n!\mathcal{E}_{0}(u)=\frac{1}{(n+1)}\int_{\Omega}u(dd^{c}u)^{n}\label{eq:energy on domain}\end{equation}
is the usual generalization to $\C^{n}$ of (minus) the squared Dirichlet
norm in the unitdisc. In the paper \cite{au} Aubin claims that the
conjectured inequality holds in the setting of the unit-ball in $\C^{n},$
but it appears that he only proved proved this under radial symmetry
(\cite{au2}, Cor 8.3 in) and in fact with a non-optimal constant
(as explained in section \ref{sec:Remarks-on-the}). Assuming only
circular symmetry, i.e. invariance under the diagonal $S^{1}-$action
on $\C^{n},$ our method of proof of Theorem \ref{thm:aubin hyp intro}
also yields the following generalization of Moser's inequality on
the disc:
\begin{thm}
\label{thm:sharp m-t in ball under inv intro}The following Moser-Trudinger
inequality holds for any $S^{1}-$invariant function in $\mathcal{H}_{0}(\mathcal{B}),$
where $\mathcal{\mathcal{B}}$ is the unit-ball in $\C^{n}:$ \begin{equation}
\log\int_{\mathcal{B}}e^{-u}dV\leq\frac{1}{(n+1)^{(n+1)}}\int_{\Omega}(-u)(dd^{c}u)^{n}+C_{n}\label{eq:m-t in thm sharp intro}\end{equation}
for a constant $C_{n}.$ Moreover the multiplicative constant in the
inequality is sharp. 
\end{thm}
Note that the sharp constant in \ref{eq:m-t in thm sharp intro} coincides
with the well-known one in the Fano setting when $X=\P^{n}$ and $k=1$
(and our proof shows that this is no coincidence). We conjecture that
the symmetry assumption in the previous theorem may be removed. In
this direction we will prove the following quasi-sharp Moser-Trudinger
inequality for a general pseudoconvex domain (or more generally a
hyperconvex one):
\begin{thm}
\label{thm:quasi-sharp m-t in ball intro}Let $\Omega$ be a pseudoconvex
domain in $\C^{n}$ with smooth boundary. Then, for any $\delta>0$
there is a constant $C$ such

\begin{equation}
\log\int_{\Omega}e^{-u}dV\leq\frac{1+\delta}{(n+1)^{(n+1)}}\int_{\Omega}(-u)(dd^{c}u)^{n}-(n-1)\log\delta+C_{n}\label{eq:quasi-sharp m-t intro}\end{equation}
 for any function $u$ in $\mathcal{H}_{0}(\Omega).$ Moreover, for
any domain $\Omega$ the limiting multiplicative constant $\frac{1}{(n+1)^{(n+1)}}$
is sharp. In particular, for any $\delta>0$ there is a constant $C_{\delta}$
such that \textup{\[
\int_{\Omega}e^{(1-\delta)n(-u)^{(n+1)/n}}dV\leq C_{\delta}\]
for any }$u$ in $\mathcal{H}_{0}(\Omega)$ such that $\int_{\Omega}(-u)(dd^{c}u)^{n}=1.$ 
\end{thm}
The proof of the latter theorem is completely different than the previous
one. The starting point is the observation that if the sharp Moser-Trudinger
inequality holds in dimension $n-1$ then so does the following sharp
\emph{Brezis-Merle type inequality:} \begin{equation}
\int_{\Omega}e^{-u}dV\leq A\left(1-\frac{1}{n^{n}}\mathcal{M}(u)\right)^{-1}\label{eq:sharp b-m intro}\end{equation}
for any $u$ in $\mathcal{H}_{0}(\Omega)$ such that $\mathcal{M}(u)^{1/n}<n,$
where $\mathcal{M}(u)$ is the total Monge-Ampère mass of $u:$ \begin{equation}
\mathcal{M}(u):=\int_{\Omega}(dd^{c}u)^{n}\label{eq:ma mass}\end{equation}
(see \cite{br-m} for the case when $n=1$ and its relation to blow-up
analysis of PDEs). We then show that, conversely a quasi-sharp version
of the Brezis-Merle inequality in dimension $n$ implies the quasi-sharp
Moser-Trudinger inequality above in the same dimension $n$ and Theorem
\ref{thm:quasi-sharp m-t in ball intro} then follows directly from
induction over $n.$ More precisely, the induction argument gives
the following quasi-sharp version of the conjectural Brezis-Merle
type inequality above.
\begin{thm}
\label{thm:quasi-sharp b-m intro}Let $\Omega$ be a pseudoconvex
domain in $\C^{n}$ with smooth boundary, where $n>1.$ Then there
is a constant $A$ such

\begin{equation}
\int_{\Omega}e^{-u}dV\leq A\left(1-\frac{1}{n^{n}}\mathcal{M}(u)\right)^{-(n-1)}\label{eq:quasi-sharp m-t intro-1}\end{equation}
 for any function in $\mathcal{H}_{0}(\Omega)$ such that $\mathcal{M}(u)^{1/n}<n.$
\end{thm}
In particular, this proves the sharp inequality in the case when $n=2.$

\subsubsection{Applications to Monge-Ampère equations}

In section \ref{sec:Existence-of-optimizers} we consider the problem
of finding extremals for Moser-Trudinger type functionals that are
parametrized by the multiplicative constants in the corresponding
inequalities. In particular, we obtain solutions to the Euler-Lagrange
equations for these functionals which are Monge-Ampère equations with
exponential non-linearities.  In the settings of domains we obtain
the following
\begin{thm}
\label{thm:mean field intro}Let $\Omega$ be a pseudoconvex domain
in $\C^{n}$ and assume that $a<(n+1)^{n}.$ Then there exists $u\in\mathcal{C}^{0}(\bar{\Omega})$
solving the equation \textup{\[
(dd^{c}u)^{n}=a\frac{e^{-u}dV}{\int_{\Omega}e^{-u}dV}\,\,\mbox{in\,\,\ensuremath{\Omega}}\,\,\,\, u=0\,\,\mbox{on\,\ensuremath{\partial\Omega}}\]
} such that $u$ optimizes the corresponding Moser-Trudinger type
functional.
\end{thm}
Here $(dd^{c}u)^{n}$ refers to the usual notion of Monge-Ampère measure
in pluripotential theory introduced by Bedford-Taylor. In the case
when $n=1$ these equations are often called \emph{mean field equations}
in the literature, as they appear in a statistical mechanical context
\cite{clmp,k} (see section \ref{sub:mf eq domain}). We also establish
a {}``concentration/compactness'' principle for the behavior of
the solutions $u_{a}$ above when $a$ approaches the critical value
$(n+1)^{n}$ (see Theorem \ref{thm:conc-comp} for the precise statement).
In particular, it implies that if there is no blow-up point in the
boundary of $\Omega,$ then (after passing to a subsequence) \emph{either
}$u_{a}$ converges to a solution of the previous equation\emph{ or}
$u_{a}$ converges to a (weak) solution of the equation \[
(dd^{c}u)^{n}=(n+1)^{n}\delta_{z_{0}},\,\,\,\,\,\mbox{\ensuremath{u=0}}\,\mbox{\,\ on\,\ensuremath{\partial\Omega}}\]
Moreover $u$ has a minimal complex singularity exponent at $z_{0}$
\cite{d-k}. It seems natural to conjecture that $u$ coincides with
$(n+1)$ times the pluricomplex Green function with a pole at $z_{0}$
\cite{kl}. This is automatically the case when $n=1$ where there
has been rather extensive work on such {}``concentration/compactness''
principles, with various elaborations (see for example \cite{clmp,m-w}).

\subsection{\label{sub:Relation-to-previous}Relations to previous results}

\subsubsection*{The Kähler setting}

On the two-sphere the inequality in Theorem \ref{thm:aubin hyp intro}
was first shown by Moser with the sharp constant $A=1/2.$ Subsequently,
the general Riemann surface case was settled by Fontana \cite{font}
with the same sharp constant. Strictly speaking these latter inequalities
were shown to hold for\emph{ any} smooth function $u,$ under the
different (but equivalent) normalization condition $\int_{X}u\omega=0.$
Then $-\mathcal{E}_{\omega}(u)$ coincides with the usual two-homogeneous
Dirichlet energy and the growth rate with respect to $k$ can hence
be reduced, by scaling, to the case $k=1.$ It should however be emphasized
that in higher dimensions this reduction argument breaks down, since
the space $\mathcal{H}_{0}(X,\omega)$ is not preserved under scaling
with positive numbers $k.$ The sharp form of the Sobolev inequalities
on the two-sphere in \ref{eq:sob intro} was obtained by Beckner \cite{Bec}.

In the case when $X$ admits a Kähler-Einstein metric the Moser-Trudinger
inequality, for the anti-canonical class and for $k=1,$ was first
shown by Ding-Tian \cite{d-t} with $A=1/V(X)$ equal to the inverse
of the volume of $-K_{X}.$ This is the sharp constant in case $X$
admits holomorphic vector fields. More precisely, they showed that
any potential of a Kähler-Einstein metric on $X$ optimizes the corresponding
Moser-Trudinger inequality (when $dV$ is taken to depend on $\omega$
in a standard way). In case $X$ has no holomorphic vector field the
constant $A=1/V(X)$ may be improved slightly as shown in the coercivity
estimate of Phong-Song-Sturm-Weinkove \cite{p-s+} (confirming a previous
conjecture of Tian).

In the case of a general Fano manifold Ding \cite{din} obtained,
using the Green function estimate of Bando-Mabuchi, a Moser-Trudinger
inequality for all $u$ in $\mathcal{H}_{0}(X,\omega)$ with a uniform
positive lower bound $\epsilon$ on the Ricci curvature of the corresponding
Kähler metric $\omega_{u}$ (for $k=1).$ The case of Theorem \ref{thm:aubin hyp intro}
for the anti-canonical class (but possibly no Kähler-Einstein metric)
and with $k=1$ was recently shown in \cite{berm1}, building on \cite{bern2}.
The approach in \cite{berm1,bern1} will be further developed in the
present paper.

\subsubsection*{The setting of domains}

A quasi-sharp version of the Brezis-Merle type inequality \ref{eq:sharp b-m intro}
was recently shown by Åhag-Cegrell-Ko\l{}odziej-Pham-Zeriahi \cite{ce}.
More precisely it was shown that the inequality holds when raising
the bracket in \ref{eq:sharp b-m intro} to the power $n.$ However
the relation to the Moser-Trudinger inequality does not seem to have
been noted before and we use it, among other things, to slightly improve
the inequality in \cite{ce} with one power. The proof uses the {}``thermodynamical
formalism'' recently introduced in \cite{berm2} (in the Kähler setting)
and shows that the Moser-Trudinger inequality is equivalent to yet
another inequality, coinciding with the classical \emph{logarithmic
Hardy-Sobolev inequality} when $n=1.$ As explained in \cite{berm2}
the corresponding inequality in the Kähler setting amounts to the
boundedness from below of Mabuchi's K-energy functional. 

Towards the end of the writing of the present paper the preprint \cite{ce3}
appeared where the existence of solutions in Theorem \ref{thm:mean field intro}
and Moser-Trudinger inequalities is proved under the stronger assumption
that $a^{1/n}<n.$ 

Let us finally point out that Demailly \cite{de} originally showed
that a weaker version of inequality \ref{eq:sharp b-m intro}) is
equivalent to a local algebra inequality previously obtained in \cite{de fer-}
in the context of the study of birational rigidity of Fano manifolds.
This latter inequality says that \begin{equation}
\mbox{lc}(\mathcal{I})\geq n/\mbox{(e}(\mathcal{I}))^{1/n},\label{eq:local alg ineq}\end{equation}
 where $\mbox{lc}(\mathcal{I})$ is the log canonical threshold of
an ideal $\mathcal{I}$ of germs of holomorphic functions and $\mbox{e}(\mathcal{I})$
is its Samuel multiplicity.

\subsection{\label{sub:Out-line-of-the}Outline of the proof of Theorems \ref{thm:aubin hyp intro},
\ref{thm:uniform fano intro}}

As is well-known a Kähler form $\omega$ is integral precisely when
it can be realized as the (normalized) curvature form of a metric
$h$ on an ample line bundle $L\rightarrow X.$ Abusing notation slightly
this means that \[
\omega=dd^{c}\phi_{0}\]
where $h=e^{-\phi_{0}}$ is the expression of the metric $h$ wrt
a local holomorphic frame. Hence, $\omega_{u}$ is the curvature form
of the metric on $L$ with weight $\phi:=\phi_{0}+u.$ The proof of
Theorem \ref{thm:aubin hyp intro} follows the same outline as the
proof of the Moser-Trudinger inequality in \cite{berm1,bern2} concerning
the case when $L=-K_{X}$ and $\phi_{0}$ is the weight of a Kähler-Einstein
metric - with some important modifications. The proof in \cite{berm1,bern2}
is based on consideration of the functional \[
\mathcal{G}(\phi):=\log\int_{X}e^{-\phi}+\frac{1}{V}\mathcal{E}(\phi,\phi_{0}),\]
 where we have used that $e^{-\phi}$ defines a global volume form
on $X$ (since $L=K_{X})$ and where $\mathcal{E}(\phi,\phi_{0}):=\mathcal{E}_{\omega}(\phi-\phi_{0}).$
The Moser-Trudinger inequality says that $\mathcal{G}$ is negative
on the space \emph{$\mathcal{H}(-K_{X})$} of positively curved metrics
on $-K_{X}$. But $\mathcal{G}$ is geodesically concave on the space
\emph{$\mathcal{H}(-K_{X})$} equipped with the Mabuchi metric (see
the next section) and the Kähler-Einstein condition says that $\phi_{0}$
is a critical point of $\mathcal{G}$. Moreover, by definition $\mathcal{G}$
vanishes at $\phi=\phi_{0}$ and that ends the proof.

At first glance, not much of this argument works in our situation
of a general line bundle $L\rightarrow X$. The functional \[
\phi\mapsto\log\int_{X}e^{-(\phi-\phi_{0})}dV\]
 has no obvious concavity properties and we have in general nothing
that corresponds to the Kähler-Einstein condition. To handle the lack
of concavity, we use a different functional, defined for each point
$x$ in $X:$ \[
\phi\mapsto\log(K_{\phi_{0}}(x)/K_{\phi}(x)),\]
where $K_{\phi}$ is the restriction to the diagonal of the Bergman
kernel for the space of global sections $H^{0}(X,L+K_{X})$ of the
adjoint line bundle $L+K_{X}$, which is known to be concave by the
results in \cite{bern1,bern1b}. It then turns out that we can replace
the Kähler-Einstein condition by a standard estimate for the Bergman
kernel in terms of the volume form; see \cite{berm1} where a similar
argument was used. The remaining problem is then to get from an estimate
of the Bergman kernel to an estimate of the metric on $L$ itself.
On a compact manifold, this can be done using the basic formula \[
\int_{X}K_{\phi}(x)e^{-\phi}=N\]
where $N$ is the dimension of $H^{0}(X,K_{X}+L)$. The growth rate
in $k$ in the inequality of the theorem is a consequence of a the
Bergman kernel estimate, using that $k\phi$ is the weight of a metric
on the $k$ th tensor power of $L,$ written as $kL$ in our additive
notation.

As for Theorem \ref{thm:uniform fano intro} it is proved by noting
that the Bergman kernel estimate can be made to be uniform over all
Fano manifolds of the same dimension by picking a reference metric
$\phi_{0}$ whose curvature form has a universal lower bound on its
Ricci curvature.

\subsection*{Acknowledgments}

It is a pleasure to thank Sébastien Boucksom, Phillipe Eyssidieux,
Vincent Guedj and Ahmed Zeriahi for the stimulation coming from \cite{bbegz}.
Also thanks to Yuji Odaka for pointing out to us that Yuji Sano had
noted that Aubin's conjecture for the optimal constant in the setting
of Fano manifolds cannot be correct.

\subsection{\label{sub:Notation-and-preliminaries}Notation and preliminaries}

Here we will briefly recall the notions of (quasi-) psh functions
and finite energy spaces in setting of compact manifolds $X$ and
domains $\Omega.$ In practice, it will, by approximation, be enough
to prove the inequalities we will be interested in for smooth (or
bounded) functions. However, the finite energy spaces play an important
role in the variational approach used in section \ref{sec:Existence-of-optimizers}.

\subsubsection*{The setting of a compact manifold $X$}

Let $(X,\omega)$ be a compact complex manifold and $\omega$ a smooth
real closed $(1,1)-$form on $X$ such that $\omega\geq0.$ We will
mainly be concerned with the case when $\omega>0,$ i.e. when $(X,\omega)$
is a Kähler manifold. Denote by $PSH(X,\omega)$ be the space of all
$\omega-$psh functions $u$ on $X,$ i.e. $u\in L_{1}(X)$ and $u$
is upper-semicontinuous (usc) and \[
\omega_{u}:=\omega+\frac{i}{2\pi}\partial\bar{\partial}u:=\omega+dd^{c}u\geq0,\]
 in the sense of currents (the normalizations are made so that $dd^{c}\log|z|^{2}=1$
when $n=1).$ We will write $\mathcal{H}(X,\omega)$ for the interior
of $PSH(X,\omega)\cap\mathcal{C}^{\infty}(X)$ (called the space of
Kähler potentials when $\omega>0)$ and $\mathcal{H}_{0}(X,\omega)$
for its subspace defined by the normalization $\sup_{X}u=0.$ We will
also use the (non-standard) notion $\mathcal{H}(X,\omega)_{b}:=PSH(X,\omega)\cap L^{\infty}(X)$
for the\emph{ b}ounded functions in $PSH(X,\omega).$ By the local
theory of Bedford-Taylor the Monge-Ampere operator \[
MA(u):=\omega_{u}^{n}/n!\]
 is well-defined on $\mathcal{H}(X,\omega)_{b}$ and continuous under
sequences decreasing to elements in $\mathcal{H}(X,\omega)_{b}$ as
are all powers $\omega_{u}^{p}.$ In particular, the functional $\mathcal{E}_{\omega}$
(formula \ref{eq:energy intro}) is well-defined and continuous in
the previous sense. Following \cite{begz,bbgz} $\mathcal{E}_{\omega}$
may be extended to all of $PSH(X,\omega)$ by setting \[
\mathcal{E}_{\omega}(u):=\inf_{v\in\mathcal{H}(X,\omega)_{b},\, v\geq u}\mathcal{E}_{\omega}(v)\in[-\infty,\infty[\]
 Now the space $\mathcal{E}^{1}(X,\omega)$ of all $\omega-$psh functions
of \emph{finite energy }may be defined as the set of all $u$ such
that $\mathcal{E_{\omega}}(u)>-\infty.$ As explained in \cite{begz,bbgz}
it coincides with the space with the same name introduced in \cite{g-z2}.

\subsubsection*{Metrics/weights on a line bundle vs. $\omega-$psh functions}

In the integral case, i.e. when $[\omega]=c_{1}(L)$ for a holomorphic
line bundle $L\rightarrow X,$ the space $PSH(X,\omega)$ may be identified
with the space of (singular) Hermitian metrics on $L$ with positive
curvature current. More precisely, let $s$ be a trivializing local
holomorphic section of $L,$ i.e. $s$ is non-vanishing an a given
open set $U$ in $X.$\emph{ }First we identify an Hermitian metric
$h_{0}=\left\Vert \cdot\right\Vert $ on $L$ with its \emph{weight
}$\phi,$ which is locally defined by the relation \[
\left\Vert s\right\Vert ^{2}=e^{-\phi_{0}}\]
The (normalized) curvature $\omega$ of the metric is the globally
well $(1,1)-$current defined by the following local expression: \[
\omega=dd^{c}\phi_{0}\]
The identification with $PSH(X,\omega)$ referred to above is now
obtained by fixing $\phi_{0}$ and letting $\phi\mapsto u:=\phi-\phi_{0}$
so that $dd^{c}\phi=\omega_{u}.$ We will denote by $\mathcal{H}_{L}$
the space of all semi-positively curved metrics/weights on $L.$

\subsubsection*{The setting of a domain $\Omega$ in $\C^{n}$ }

Let $\Omega$ be a bounded domain $\C^{n}$ (in this setting $\omega=0)$
which is \emph{hyperconvex}, i.e. it admits a negative continuous
psh exhaustion function (for example a pseudoconvex domain with Lipschitz
continuous boundary) The main reason that we will consider general
hyperconvex domains (with possible non-smooth boundary) is that this
property is preserved under Cartesian products. When $\Omega$ has
smooth boundary we let $\mathcal{H}_{0}(\Omega)$ be the subspace
of all smooth psh functions on $\bar{\Omega}$ such that $u=0$ on
$\partial\Omega.$ Following \cite{ce0,ce} (see also \cite{b-g-z}
for a comparison with the Kähler setting) it will also be convenient
to use two singular versions of $\mathcal{H}_{0}(\Omega),$ namely
$\mathcal{F}(\Omega)$ and $\mathcal{E}_{1}(\Omega),$ where the Monge-Ampère
mass $\mathcal{M}(u)$ \ref{eq:ma mass} and energy $\mathcal{E}_{0}(:=\mathcal{E})$
\ref{eq:energy on domain} are well-defined and finite, respectively.
More precisely, let first $\mathcal{H}_{0}(\Omega)_{b}$ be the space
all $u$ in $PSH\ensuremath{(\Omega)\cap L^{\infty}(\Omega)}$ such
that $\mathcal{M}(u)<\infty$ and such that $\lim_{\zeta\rightarrow z}u(z)=0$
for any $z\in\partial\Omega$ (called the space of psh {}``test-functions''
$\mathcal{E}_{0}(\Omega)$ in \cite{ce0}). Now $\mathcal{F}(\Omega)$
is defined as the space of all $u$ such that there exists $u_{j}\in\mathcal{H}_{0}(\Omega)_{b}$
decreasing to $u$ with $\mathcal{M}(u_{j})\leq C.$ The Monge-Ampère
operator extends to $\mathcal{E}_{0}(\Omega)$ in is continuous under
decreasing limits. As for the space $\mathcal{E}_{1}(\Omega)$ it
is defined in a similar manner, but by demanding that $-\mathcal{E}(u_{j})\leq C.$
There is also an alternative characterization of $\mathcal{F}(\Omega)$
as the set of all $u$ in the {}``domain of definition of the Monge-Ampère
operator'' such that $u$ has finite total Monge-Ampère mass and
with smallest maximal plurisubharmonic majorant equal to zero (see 

For the purpose of the present paper it will in practice be enough
to know that if $u\in PSH\ensuremath{(\Omega)\cap L^{\infty}(\Omega)}$
such that $\lim_{\zeta\rightarrow z}u(z)=0$ for any $z\in\partial\Omega,$
then $u\in\mathcal{F}(\Omega)$ if $\int_{\Omega}(dd^{c}u)^{n}<\infty$
and similarly $u\in\mathcal{E}^{1}(\Omega)$ if $\int_{\Omega}(-u)(dd^{c}u)^{n}<\infty$
(see \cite{ce0,ce}). 

It may also be convenient to recall (even if, strictly speaking, it
will not be needed) the approximation result in \cite{ce2} saying
that \emph{any} negative psh function $u$ on a hyperconvex domain
$\Omega$ can be written as decreasing limit of {}``smooth test functions'',
i.e. psh functions $u_{j}$ in $\mathcal{C}(\bar{\Omega})\cap\mathcal{C}^{\infty}(\Omega),$
vanishing on the boundary and with finite Monge-Ampère mass. As a
consequence one may as well replace the space $\mathcal{H}_{0}(\Omega)_{b}$
in the previous definitions with the space of {}``smooth test functions''
in the previous sense.

\section{Moser-Trudinger inequalities on Kähler manifolds}

Let $X$ be an $n$-dimensional compact Kähler manifold and let $L$
be a semipositive line bundle over $X$ and assume that $L$ is big,
i.e. \[
V=\int_{X}(dd^{c}\phi)^{n}/n!>0\]
for any (and hence all) $\phi$ in $\mathcal{H}(L).$ We fix $\phi_{0}\in\mathcal{H}(L)$
and let $\omega:=dd^{c}\phi_{0}.$

\subsection{Energy, geodesics and Bergman kernels (preliminaries)}

Given $\phi$ and $\phi_{0}$ in $\mathcal{H}(L)$ we define (minus)
the relative energy by \[
\mathcal{E}(\phi,\phi_{0})=\frac{1}{(n+1)!}\int_{X}(\phi_{0}-\phi)\sum_{0}^{n}(dd^{c}\phi_{0})^{k}\wedge(dd^{c}\phi)^{n-k}\]
If $t\rightarrow\phi_{t}$ is a smooth curve in $\mathcal{H}(L)$
and \[
\dot{\phi}_{t}:=\frac{d\phi_{t}}{dt}\]
 then \[
\frac{d}{dt}\mathcal{E}(\phi_{t},\phi_{0})=\int_{X}\dot{\phi}_{t}(dd^{c}\phi_{t})^{n}/n!.\]
 This formula, together with the normalization $\mathcal{E}(\phi_{0},\phi_{0})=0$
can also be used to define $\mathcal{E}$.

A basic property of $\mathcal{E}$ is that it is linear along \textit{geodesics}
in $\mathcal{H}(L)$ and concave along subgeodesics defined wrt Mabuchi's
Riemannian metric on $\mathcal{H}(L).$ For technical reasons we will
work with the following weaker notion of geodesics. Given two smooth
metrics $\phi_{0}$ and $\phi_{1}$ the corresponding geodesic $\phi_{t}$
is defined as the following regularized envelope: \[
\phi_{t}:=\Phi(z,t):=\sup_{\mathcal{\psi\in K}}\left\{ \Psi(z,t)\right\} ^{*}\]
where we have extended $t$ to the strip $\mathcal{T=}[0,1]+i\R$
in $\C$ and $\mathcal{K}$ is the set of all semi-positively curved
metrics $\Psi$ on the pull-back of $L$ to $X\times\mathcal{T}$
such that $\psi_{0}\leq\phi_{0}$ and $\psi_{0}\leq\phi_{1}.$ We
will sometimes refer to a curve $\psi_{t}:=\Psi(\cdot,t)$ above as
a\emph{ subgeodesic}. When $L$ is ample it was shown in \cite{b-d}
that $\Psi$ is a continuous solution to the Dirichlet problem for
the Monge-Ampère operator on $M:=X\times\mathcal{T},$ i.e. \[
(dd^{c}\Phi)^{n+1}=0\]
 in the interior of $M$ (in the usual sense of pluripotential theory)
and on the boundary $\partial M$ the metric $\Phi$ coincides with
the $i\R$ invariant boundary data determined by $\phi_{0}$ and $\phi_{1}.$
However, we will only need some very modest regularity properties
of $\Phi,$ namely that $\Phi$ is locally bounded and that $\Phi(t,\cdot)=\phi_{t}$
converges uniformly to the given boundary data as $t$ approaches
$\partial\mathcal{T}$. As shown by a simple barrier argument this
is always the case as long as $L$ is semi-positive (see \cite{bern2}).
Indeed, \begin{equation}
\chi_{t}:=\max\{\phi_{0}-A\Re t,\phi_{1}-A(1-\Re t)\}\label{eq:barrier in compact case}\end{equation}
gives a candidate for the sup defining $\phi_{t}$ converging uniformly
towards the right boundary values. Hence so does $\phi_{t}.$ Also
note that, by imposing $S^{1}-$symmetry in the complex variable $t$
we might as well replace $\mathcal{T}$ with an annulus $\mathcal{A}.$ 
\begin{lem}
\label{lem:prop of energy kahler}Let $\phi_{t}$ be a (weak) geodesic
as above. Then $t\mapsto\mathcal{E}(\phi_{t},\phi_{0})$ is affine
and continuous up to the boundary of $[0,1].$ Moreover, if\textup{
$\dot{\phi}_{0}$ denotes the right derivative of $\phi_{t}$ at $t=0$
(which exists by convexity), then} \textup{\[
\frac{d}{dt}_{t=0^{+}}\mathcal{E}(\phi_{t})\leq\int_{\mathcal{B}}\dot{\phi}_{0}(dd^{c}\phi_{0})^{n}/n!,\]
}
\end{lem}
As pointed out above this is well-known in the case when $\phi_{t}$
is smooth and follows immediately from the formula \begin{equation}
d_{t}d_{t}^{c}\mathcal{E}(\phi_{t},\phi_{0})=\int_{X}(dd^{c}\Phi)^{n+1}\label{eq:second deriv of energy}\end{equation}
 The general case then follows by approximation (see \cite{berm1});
see also Prop \ref{lem:prop of energy along geod in domain} for the
corresponding properties in the setting of domains. 

Any element $\phi$ in $\mathcal{H}(L)$ defines an $L^{2}$ metric
on $H^{0}(X,K_{X}+L)$, \[
\|u\|_{\phi}^{2}=i^{n^{2}}\int u\wedge\bar{u}e^{-\phi}.\]
 The Bergman kernel for this $L^{2}$-metric is denoted $K_{\phi}(x)$.
It can be defined as in the introduction \[
K_{\phi}(x)=i^{n^{2}}\sum u_{j}(x)\wedge\bar{u}_{j}(x)\]
 where $u_{j}$ is an orthonormal basis for $H^{0}(X,K_{X}+L)$. Alternatively,
\begin{equation}
K_{\phi}(x)=\sup_{H^{0}(X,K_{X}+L)}\,\,\{|u(x)|^{2};\|u\|_{\phi}\leq1\}.\label{eq:extremal def}\end{equation}
 Here the expression $|u(x)|^{2}$ depends on the choice of a trivialization
of $L$ near $x$, but $\log K_{\phi}$ is invariantly defined as
a metric on $K_{X}+L$. As a consequence, the quotient of two Bergman
kernels \[
K_{\phi}(x)/K_{\phi_{0}}(x)\]
 is a global function on $X$, smooth if the sections in $H^{0}(X,K_{X}+L)$
have no common zeros.

We will use a result from \cite{bern1} saying that \[
t\rightarrow\log K_{\phi_{t}}(x)\]
 is, for any $x$ fixed, convex along (sub)geodesics $\phi_{t}.$

The first result we will need is the following simple formula for
the derivative of the Bergman kernel along a curve (see for example
the appendix in \cite{berm1}). 
\begin{lem}
\label{lem:der of bergman}Let $\phi_{t}$ be a smooth curve in $\mathcal{H}(L).$
Then \[
\frac{d}{dt}K_{\phi_{t}}(x)=\int_{X}\dot{\phi}_{t}|K_{\phi_{t}}(x,y)|^{2}e^{-\phi_{t}}\]
 where the off-diagonal Bergman kernel is \[
K_{\phi_{t}}(x,y):=\sum c_{n}u_{j}(x)\wedge\bar{u}_{j}(y)\]
 for any orthonormal basis of $H^{0}(X,K_{X}+L)$. 
\end{lem}

\subsection{Moser-Trudinger type inequalities }

The next proposition is the crux of the proof of the Moser-Trudinger
inequalities.
\begin{prop}
Let $\phi$ and $\phi_{0}$ be two metrics in $\mathcal{H}(L)$, satisfying
the normalizing condition \[
\phi-\phi_{0}\leq0.\]
 Assume that the Bergman kernel for $\phi_{0}$ satisfies

\begin{equation}
K_{\phi_{0}}e^{-\phi_{0}}\leq C_{1}(dd^{c}\phi_{0})^{n}/n!\label{eq:def of c1}\end{equation}
Then \[
\sup_{X}\log\frac{K_{\phi_{0}}}{K_{\phi}}\leq C_{1}\mathcal{E}(\phi_{t},\phi_{0}).\]
\end{prop}
\begin{proof}
Join $\phi_{0}$ and $\phi$ with a geodesic $\phi_{t}$ such that
$\phi_{1}=\phi$. By the previous lemma \[
-\frac{d}{dt}|_{t=0}\log K_{\phi_{t}}(x)=\int_{X}-\dot{\phi}_{0}\frac{|K_{\phi_{0}}(x,y)|^{2}}{K_{\phi_{0}}(x)}e^{-\phi_{0}}.\]
 Since $\phi_{t}$ is a geodesic, $\phi_{t}$ is convex in $t$, so
\[
\dot{\phi}_{0}\leq\phi-\phi_{0}\leq0.\]
 Hence, since by Cauchy's inequality \[
|K_{\phi_{0}}(x,y)|^{2}\leq K_{\phi_{0}}(x)K_{\phi_{0}}(y),\]
 \[
-\frac{d}{dt}|_{t=0}\log K_{\phi_{t}}(x)\leq\int_{X}-\dot{\phi}_{0}K_{\phi_{0}}(y)e^{-\phi_{0}},\]
 which in turn is dominated by \[
C_{1}\int_{X}-\dot{\phi}_{0}(dd^{c}\phi_{0})^{n}/n!\leq C_{1}\frac{d}{dt}|_{t=0}\mathcal{E}(\phi_{t},\phi_{0})\]
 by the definition of $C_{1}$ (formula \ref{eq:def of c1}) and Lemma
\ref{lem:prop of energy kahler} which also gives \[
\frac{d}{dt}|_{t=0}\mathcal{E}(\phi_{t},\phi_{0})=\mathcal{E}(\phi,\phi_{0}).\]
 Now we use that $f(t):-\log K_{\phi_{t}}$ is concave. Therefore
\[
f(1)-f(0)\leq f'(0)\]
 which means that \[
\log K_{\phi_{0}}-\log K_{\phi}\leq f'(0)\leq C_{1}\mathcal{E}(\phi,\phi_{0})\]
 which completes the proof.
\end{proof}
Now it only remains to convert this estimate of the Bergman kernel
to an estimate of the integral of $e^{-\phi}$. Here we use \begin{equation}
\int_{X}K_{\phi}e^{-\phi}=N:=\dim H^{0}(X,L+K_{X})\label{eq:integral of bergman is dim}\end{equation}
Let $C_{1}$ and $C_{2}$ be constants satisfying \begin{equation}
C_{2}dV\leq K_{\phi_{0}}e^{-\phi_{0}}\leq C_{1}(dd^{c}\phi_{0})^{n}/n!\label{eq:upper lower bounds on bergman}\end{equation}
where $dV$ is a fixed volume form on $X.$ Note that $L+K_{X}$ is
basepoint free precisely when $C_{2}$ can be taken to be strictly
positive.

By the previous proposition and \ref{eq:upper lower bounds on bergman}
we have for any $x$ in $X$ \begin{equation}
K_{\phi}\geq K_{\phi_{0}}e^{-C_{1}\mathcal{E}(\phi,\phi_{0})}\geq C_{2}e^{\phi_{0}}dVe^{-C_{1}\mathcal{E}(\phi,\phi_{0})},\label{eq:lower bd on bergman}\end{equation}
 so it follows that \[
\int_{X}e^{-(\phi-\phi_{0})}dV\leq C_{2}^{-1}Ne^{-C_{1}\mathcal{E}(\phi,\phi_{0})}.\]
 We collect this in the next theorem which, as explained below, implies
Theorem \ref{thm:aubin hyp intro} in the introduction.
\begin{thm}
\label{thm:main in text}Let $\phi_{0}$ be a semipositively curved
metric on the line bundle $L$ over the compact Kähler manifold $X$.

Assume that the Bergman kernel for $\phi_{0}$ satisfies \ref{eq:upper lower bounds on bergman}.
Then for any other semipositively curved metric on $L$, satisfying
\[
\phi-\phi_{0}\leq0.\]
 we have that \[
\log\int_{X}e^{-(\phi-\phi_{0})}dV\leq\log(N/C_{2})-C_{1}\mathcal{E}(\phi,\phi_{0}).\]
 
\end{thm}
\bigskip{}

We say ( cf \cite{bern1b},\cite{berm1}) that the metric $\phi_{0}$
is \textit{balanced in the adjoint sense} if there is a constant $C$
such that \[
K_{\phi_{0}}e^{-\phi_{0}}=C(dd^{c}\phi_{0})^{n}/n!.\]
 When $dV:=(dd^{c}\phi_{0})^{n}/n!$ this amounts to saying that the
constants $C_{1}$ and $C_{2}$ in \ref{eq:upper lower bounds on bergman}
can be chosen equal, and integrating over $X$ we see that in this
case $C=N/V$. We thus immediately get the next corollary. 
\begin{cor}
\label{cor:balanc}With assumptions as in Theorem 2.3, assume in addition
that $\phi_{0}$ is balanced in the adjoint sense. Then \[
\log\int_{X}e^{-(\phi-\phi_{0})}(dd^{c}\phi_{0})^{n}/n!\leq\frac{N}{V}\mathcal{E}(\phi,\phi_{0}).\]
 
\end{cor}
As an example of this, let us look at the case $L=-K_{X}$. Then $H^{0}(X,K_{X}+L)=\C,$
i.e. $N=1,$ and \[
K_{\phi_{0}}(x)=1/\int_{X}e^{-\phi_{0}}.\]
 Hence the condition that $\phi_{0}$ be balanced in the adjoint sense
means that \[
(dd^{c}\phi_{0})^{n}/(Vn!)=(\int_{X}e^{-\phi_{0}})^{-1}e^{-\phi_{0}}\]
 which means that $\phi_{0}$ is the potential of a Kähler-Einstein
metric. Then the corollary becomes \[
\log\int_{X}e^{-\phi}\leq\log\int_{X}e^{-\phi_{0}}+\mathcal{E}(\phi,\phi_{0})\]
 since $N=1$. This is the Moser-Trudinger inequality first proved
in \cite{d-t} (using a different method). Note that the assumption
that $\phi\leq\phi_{0}$ is unnecessary here since both sides scale
the same way if we subtract a constant from $\phi$.

\bigskip{}

\subsubsection{Proof of Theorem \ref{thm:aubin hyp intro}}

In a similar vein we can consider asymptotic versions of Theorem 2.3,
when we replace $L$ be $kL$, with $k$ a large integer. Then it
follows from well-known Bergman kernel asymptotics due to Bouche and
Tian (see \cite{z} and references therein for various refinements)
that for any fixed smooth and strictly positively curved $\phi_{0}$
\begin{equation}
K_{k\phi_{0}}e^{-k\phi_{0}}=(dd^{c}k\phi_{0})^{n}/n!(1+O(k^{-1}))\label{eq:bergman kernel as}\end{equation}
Hence in \ref{eq:upper lower bounds on bergman} we can take both
$C_{1}$ and $C_{2}$ equal to \[
1+O(k^{-1})\]
 Integrating \ref{eq:bergman kernel as} we also get the well known
formula \[
N_{k}=Vk^{n}+o(k^{n-1})\]
 for the dimension of the space of global sections of $K_{X}+kL$.
Altogether this finishes the proof of Theorem \ref{thm:aubin hyp intro}.

\subsection{Uniformity over all Fanos (proof of Theorem \ref{thm:uniform fano intro})}

We start with the following essentially well-known lemma which is
proved using Moser iteration (see Thm. 7 in \cite{li} which is stated
for the equality case in \ref{eq:moser iteration bound}, but the
proof in general is the same) 
\begin{lem}
Let $(X,g)$ be a Riemannian manifold of real dimension $2n>2$ and
let $a_{g}$ and $b_{g}$ be constants such that the following Sobolev
inequality holds for any function $F$ on $X$ such that $F$ and
its gradient are  in $L^{2}:$ \[
(\int_{X}|F|^{2\sigma}dV_{g})^{1/\sigma}\leq\left(a_{g}\int_{X}|\nabla_{g}F|^{2}dV_{g})+b_{g}\int_{X}|F|^{2}dV_{g}\right),\,\,\,\sigma=n/(n-1)\]
For any positive function $H$ such that $\Delta_{g}H\geq-\lambda H$
there is a constant $C{}_{g}$ only depending on $a_{g}$ and $b_{g}$
such that \begin{equation}
\left\Vert H\right\Vert _{L^{\infty}(X)}\leq C{}_{g}\lambda^{n}\left\Vert H\right\Vert _{L^{1}(X,g)}\label{eq:moser iteration bound}\end{equation}

\end{lem}
Let us now assume that $L\rightarrow X$ is a an ample line bundle
with a fixed positively curved weight $\phi_{0}$ such that the Kähler
form $\omega_{0}:=dd^{c}\phi_{0}$ has a lower bound $\delta$ on
its Ricci curvature: \begin{equation}
\mbox{Ric }\omega_{0}\geq\delta\omega_{0}\label{eq:lower bd ricci}\end{equation}
Then we claim that there is a constant $C_{\delta}$ only depending
on $\delta$ such that the Bergman kernel $K_{k\phi_{0}}(x)$ of the
space $H^{0}(kL+K_{X})$ has the following point-wise upper bound:
\begin{equation}
K_{k\phi_{0}}\leq C_{\delta}k^{n}(dd^{c}\phi_{0})^{n}/n!\label{eq:upper bd bergman by ricci}\end{equation}
To see this let $g$ be the Riemannian metric on $X$ corresponding
to $\omega_{0}.$ By \cite{il} the corresponding constants $a_{g}$
and $b_{g}$ only depend on the lower bound $\delta$ of the Ricci
curvature of $g$ the lower bound on the volume $V_{g}$ of $g:$
\[
a_{g}:=\frac{2n-1}{n(n-1)\delta V^{1/n}},\,\,\, b_{g}=\frac{1}{V^{1/n}}\]
 Let now $f_{k}$ be an element in $H^{0}(kL+K_{X})$ and write \[
H:=|f_{k}|^{2}e^{-k\phi_{0}}/((dd^{c}\phi_{0})^{n}/n!)\]
Then it follows immediately from the definition of Ricci curvature
and the fact that $\log|f_{k}|^{2}$ is locally psh that \[
dd^{c}\log H\geq-k\omega_{0}-\delta\omega_{0}\]
 and hence $dd^{c}H\geq-(k-\delta)H\omega_{0}.$ Applying the previous
Lemma to $H$ with $\lambda:=n(k-\delta)$ now gives \[
|f_{k}|^{2}e^{-k\phi_{0}}\leq C\delta k^{n}\frac{(dd^{c}\phi_{0})^{n}}{n!}\int_{X}|f_{k}|^{2}e^{-k\phi_{0}}.\]
 By the extremal definition of $K_{k\phi_{0}}$ this finally proves
the inequality \ref{eq:upper bd bergman by ricci}.

Let us now assume that $X$ is a Fano manifold and take $L:=-K_{X}$
so that $V:=c_{1}(-K_{X})^{n}/n!.$ As shown by Tian-Yau \cite{t-y}
one may always choose $\omega:=\omega_{0}\in c_{1}(-K_{X})$ so that
$1/\delta$ in \ref{eq:lower bd ricci} only depends on an upper bound
on $V$ (since changing $\phi_{0}$ only changes the additive constant
$B_{k}$ we are allowed to choose $\phi_{0}$ and $dV$). As later
shown in \cite{ca,kmm} the volume $V$ of an $n-$dimensional Fano
has a universal bound $V\leq c_{n}$ and hence $\phi_{0}$ may be
chosen so that the Bergman kernel estimate \ref{eq:upper bd bergman by ricci}
holds with a constant $C_{\delta}$ only depending on the dimension
$n.$ The proof of Theorem \ref{thm:uniform fano intro} is now concluded
by invoking Theorem \ref{thm:main in text}.
\begin{rem}
One may also ask whether there is universal\emph{ lower }bound on
$\inf_{X}(K_{k\phi}e^{-k\phi}/(dd^{c}\phi)^{n})$ in terms of a positive
lower bound $\delta$ of the Ricci curvature of $dd^{c}\phi$ and
the dimension $n$ of the Fano manifold? If one instead considers
the Bergman kernel $\tilde{K}_{k\phi}$ defined wrt the $L^{2}-$norm
$\int_{X}|f|^{2}e^{-k\phi}(dd^{c}\phi)^{n}$ on $H^{0}(X,kK)$ then
a lower bound for $\tilde{K}_{k\phi}e^{-k\phi}$ was obtained by Tian
\cite{ti2} when $n=2$ (for all $\phi=\phi_{t}$ appearing in Aubin's
continuity path) and it forms a crucial role in the proof in \cite{ti2}
of the Calabi conjecture on Fano (Del Pezzo) surfaces. 
\end{rem}

\subsection{\label{sub:Vanishing-Lelong-numbers}Vanishing Lelong numbers (proof
of Cor \ref{cor:no lelong intro})}

Let now $\omega$ be a semi-positive form with positive volume and
$L$ a semi-positive and big line bundle (i.e. $V>0)$ with $c_{1}(L)=[\omega].$
Take as before $\phi_{0}\in\mathcal{H}_{L}$ with curvature form $\omega\geq0.$
In case $kL+K_{X}$ is semi-ample for all $k\geq k_{0}$ we have,
by definition, that $K_{k\phi_{0}}>0$ on all of $X$ and hence $C_{2}>0$
so that Cor \ref{cor:no lelong intro} follows immediately from Theorem
\ref{thm:main in text}. In the general case we proceed as follows.
Given $\phi$ a locally bounded weight with $dd^{c}\phi\geq0$  \ref{eq:lower bd on bergman}
gives

\begin{equation}
e^{C_{1}C}\int_{X}K_{k\phi_{0}}e^{-k\phi}\leq\int_{X}K_{k\phi}e^{-k\phi}=N_{k}\label{eq:proof of lelong}\end{equation}
Next, we claim that there is a holomorphic section $s_{E}$ of a holomorphic
line bundle $E\rightarrow X$ with a smooth weight $\phi_{E}$ such
that \begin{equation}
|s_{E}|^{2}e^{-\phi_{E}}/C'\leq K_{k\phi_{0}}e^{-k\phi_{0}}\label{eq:lower bd on sing berman ker}\end{equation}
for some constant $C'.$ This is an essentially well-known consequence
of the Ohsawa-Takegoshi-Manivel extension theorem proved as follows.
Since $L$ is big we may, by Kodaira's lemma, decompose $k_{0}L=A+E$
for $A$ ample (positive) and $E$ effective, i.e. $A$ admits a positively
curved weight $\phi_{A}$ and $E$ admits a holomorphic section $s_{E}.$
Then $\psi:=\phi_{A}+\log|s_{E}|^{2}$ defines a positively curved
weight on $L$ such that its curvature $dd^{c}\psi\geq dd^{c}\phi_{A}:=\omega_{A}$
is a Kähler current. We also fix a smooth weight $\phi_{K_{X}}$on
$K_{X}.$ By the Ohsawa-Takegoshi-Manivel extension theorem  $k_{0}$
may be chosen sufficiently large so that for any fixed point $x\in X$
there is a holomorphic section $f_{k}\in H^{0}(kL+K_{X})$ such that\[
|f_{k}|^{2}(x)e^{-((k-k_{0})\phi_{0}+\psi+\phi_{K_{X}}}=1,\,\,\,\int_{X}|f_{k}|^{2}e^{-((k-k_{0})\phi_{0}+\psi}\leq C''\]
for an absolute constant $C''.$ Since, we may after, subtracting
a constant, assume that $\psi\leq\phi_{0}$ we can replace the $\psi$
in the inequality above with $\phi_{0}$ and hence \ref{eq:lower bd on sing berman ker}
follows from the extremal property of the Bergman kernel.

Combining \ref{eq:proof of lelong} and \ref{eq:upper lower bounds on bergman}
now gives the existence of a constant $C$ such that \begin{equation}
\int_{X}|s_{E}|^{2}e^{\phi_{0}-\phi_{E}}e^{-k\phi}\leq CN_{k}e^{-C\mathcal{E}(\phi,\phi_{0})}\label{eq:dege mos-tr}\end{equation}
 for any fixed $k$ and $\phi$ as above with $\sup(\phi-\phi_{0})=0.$
Hence, if $\phi$ is a weight of finite energy it follows from the
definition that the integral in the lhs above is finite. But then
it follows from a simple local argument that a local representative
of $\phi$ cannot have any Lelong numbers at given point $x$ in $X.$
Indeed, if $\phi$ had Lelong number $l>0$ at $z=0$ in local coordinates
$z$ on a small ball $\mathcal{B},$ then $\phi\leq l\log|z|^{2}+C.$
If we now blow-up the point $z=0$ and denote by $s$ the section
cutting out the exceptional divisor $E_{0}$ on the blow-up $\mathcal{B}_{0}$
of $\mathcal{B}$ we get \[
\int_{\mathcal{B}_{0}}|s|^{2m}|s|^{-2(lk-(n-1))}<\infty\]
where $m$ is the order of vanishing of $s_{E}$ along $E_{0}.$ Hence,
taking $k$ sufficiently large (i.e. so that $kl>m+n-1)$ finally
yields the desired contradiction, showing that $\phi$ has no Lelong
numbers, which by Skoda's result is equivalent to $e^{-k\phi}\in L_{loc}^{1}$
for any $k\geq0$ (see for example \cite{d-k}). Finally, since $\{\mathcal{E}\geq-C\}\cap\{\sup_{X}=0\}$
is a compact set in $PSH(X,\omega)$ where along Lelong numbers vanish
identically the last statement of the corollary follows from either
from the uniform version of Skoda's theorem in \cite{zer} (just as
in the proof of Lemma 6.4 in \cite{bbgz}) or alternatively the semi-continuity
of complex integrability exponents established in \cite{d-k}.

\section{\label{sec:M-T ineq in ball under invariance-1}Moser-Trudinger inequality
in the ball under $S^{1}-$invariance}

In this section we will look at estimates for integrals of $e^{-\phi}$,
where $\phi$ is plurisubharmonic in a domain in $\C^{n}$. For simplicity
we will treat only the ball $\mathcal{B}$. As in the previous section
we let $K_{\phi}(x)$ be the Bergman kernel at the diagonal for the
plurisubharmonic weight function $\phi$. It follows from the results
in \cite{bern1} that $\log K_{\phi_{t}}(x)$ is convex in $t$ if
$t\rightarrow\phi_{t}$ is a geodesic in the space of plurisubharmonic
functions in $\B$ (see below).

We say that a function $f$ is $S^{1}$-invariant if $f(e^{i\theta}z)=f(z)$.
(Here $e^{i\theta}$ acts diagonally so that $e^{i\theta}(z_{1},...z_{n}):=(e^{i\theta}z_{1},....e^{i\theta}z_{n})$).
We also say that a domain is $S^{1}$-invariant if $e^{i\theta}z$
lies in the domain if $z$ does.

\subsection{Bergman kernels and plurisubharmonic variations}
\begin{prop}
\label{pro:bergman kernel under s one inv}Assume $\phi$ is plurisubharmonic
in an $S^{1}$ invariant (connected) domain that contains the origin
and that $\phi$ is also $S^{1}$-invariant. Then \[
K_{\phi}(0,\zeta)=1/\int e^{-\phi}\]
 for all $\zeta$ in the ball. \end{prop}
\begin{proof}
By definition, $K_{\phi}(0,\zeta)$ is antiholomorphic in $\zeta$
and by uniqueness of Bergman kernels it must also be $S^{1}$-invariant.
Hence it is a constant, and since \[
\int1K_{\phi}(0,\cdot)e^{-\phi}=1\]
 the proposition follows. 
\end{proof}
The next proposition then follows immediately from the plurisubharmonic
variation of Bergman kernels (cf \cite{bern1}). 
\begin{prop}
\label{pro:psh variation of integrals}Let $\phi_{t}$ be a subgeodesic
of $S^{1}$-invariant plurisubharmonic functions in the disk. Then
\[
t\mapsto\log(\int e^{-\phi_{t}})\]
 is concave. 
\end{prop}

\subsection{Energy and geodesics}

In this section we will adapt the results about geodesics and energy
in the compact Kähler setting to the setting of domains. In principle
all the previous properties go through in this latter setting. The
main technical difference is that one has to be a bit careful when
performing integration by parts, due to the presence of the boundary.
For this reason it will be convenient to work in the singular setting
of the finite energy class $\mathcal{E}(\Omega)$ (compare section
\ref{sub:Notation-and-preliminaries}).

In a domain $\Omega$ we have a variant of the energy $\mathcal{E}$,
which in case $\phi_{0}$ and $\phi$ are smooth is defined by \[
\mathcal{E}(\phi,\phi_{0})=\frac{1}{(n+1)!}\int_{\Omega}(\phi-\phi_{0})\sum_{0}^{n}(dd^{c}\phi_{0})^{k}\wedge(dd^{c}\phi)^{n-k}\]
 and when $\phi=\phi_{0}=0$ on $\Omega$ integration by parts show
that $\mathcal{E}(\phi,\phi_{0})=\mathcal{E}(\phi)-\mathcal{E}(\phi_{0})$
(compare the lemma below), where \[
\mathcal{E}(\phi):=\mathcal{E}_{0}(\phi):=\mathcal{E}(\phi,0)\]
 so that \[
\mathcal{E}(\phi)=\frac{1}{(n+1)!}\int_{\mathcal{B}}(-\phi)(dd^{c}\phi)^{n}.\]
Moreover, integration by parts also give \[
\frac{d}{dt}\mathcal{E}(\phi_{t},\phi_{0})=\int_{\mathcal{B}}\dot{\phi}_{t}(dd^{c}\phi_{t})^{n}/n!\]
We will need the following generalization:
\begin{lem}
\label{lem:first deriv of energy in domain}Let $\phi$ and $\psi$
be in \textup{$\mathcal{E}^{1}(\Omega)).$ Then \[
\frac{d}{dt}_{t=0^{+}}\mathcal{E}(\phi+t(\psi-\phi))=\int_{\mathcal{B}}(\psi-\phi)(dd^{c}\phi)^{n}/n!\]
Moreover, the following cocycle relation holds $\mathcal{E}(\phi)-\mathcal{E}(\psi)=\mathcal{E}(\phi,\psi).$}\end{lem}
\begin{proof}
Assume first that $\phi$ and $\psi$ are in $\mathcal{H}_{0}(\Omega)_{b}.$
In this class one may integrate by parts just as in the smooth case
(using the assumption on finite Monge-Ampère mass; see \cite{ce0}
and references therein) and hence expanding $\mathcal{E}(\phi+t(\psi-\phi))$
and integrating by parts gives\[
\mathcal{E}(\phi+t(\psi-\phi))=t\int_{\mathcal{B}}(\psi-\phi)(dd^{c}\phi)^{n}+O(t^{2})I\]
where $I$ is a sum of terms of the form $\int(\psi-\phi)(dd^{c}\phi)^{n-j}(dd^{c}\psi)^{j}$
which are finite since $\phi$ and $\psi$ are in $\mathcal{H}(\Omega)_{b}.$
This finishes the proof in the case of the class $\mathcal{H}_{0}(\Omega)_{b}.$
Finally, given $\phi$ and $\psi$ in $\mathcal{E}^{1}(\Omega))$
we take sequences $\phi_{j}$ and $\psi_{k}$ in $\mathcal{H}(\Omega)_{b},$
decreasing to $\phi$ and $\psi$ respectively. By the previous case
we have \[
\mathcal{E}(\phi_{j}+t(\psi_{k}-\phi_{j}))=\int_{0}^{t}\int_{\Omega}(\psi_{k}-\phi_{j})(dd^{c}(\phi_{j}+s(\psi_{k}-\phi_{j})))^{n}ds=:\int_{0}^{t}g_{k,j}(s)ds\]
By well-known continuity properties \cite{ce0} and the finite energy
assumptions letting first $j$ and then $k$ tend to infinity shows
that the previous formula holds with $\phi_{j}$ and $\psi_{k}$ replaced
with $\phi$ and and $\psi,$ respectively. Moreover, for the same
reason the corresponding density $g(s)$ is continuous wrt $s$ and
that ends the proof of the derivative formula in the general case.
Finally, the previous formula implies the cocycle relation by integrating
along the line $t\mapsto\phi+t(\psi-\phi)$ (note that by a well-known
Cauchy-Schwartz type estimate all terms in $\mathcal{E}(\phi,\psi)$
are finite). 
\end{proof}
Next we turn to the definition of geodesic segments in the setting
of domains. Given, say $\phi_{0}$ and $\phi_{1}$ on $\Omega$ which
are psh and smooth up to the boundary, where they vanish, the corresponding
geodesic $\phi_{t}$ is defined by replacing the space $\mathcal{H}(L)_{b}$
with the space of all bounded psh functions tending to zero at the
boundary. In other words, a geodesic is defined as the following regularized
envelope, where $M:=\Omega\times\mathcal{A}$ (with $\mathcal{A}$
denoting an annulus): \[
\phi_{t}:=\Phi(z,t):=\sup_{\mathcal{\psi\in K}}\left\{ \Psi(z,t)\right\} ^{*}\]
where $\mathcal{K}$ is the set of all psh functions $\Psi\in PSH\cap L^{\infty}(M)$
such that $\Psi^{*}\leq f$ on $\partial M,$ where $f$ is the function
on $\partial M$ defined as follows: decomposing $\partial M:=B_{1}\cup B_{2}:=\partial\Omega\times\mathcal{A\cup\Omega\times\mathcal{\partial A}}$
we let $f=0$ on $B_{1}$ and $f=\phi_{i}$ for $i=1,2$ on the two
different components of $B_{2}.$ In particular, if the$\phi_{0}$
and $\phi_{1}$ are continuous on $\bar{\Omega}$ then so is the boundary
data $f.$ Just as in the setting of compact Kähler manifolds we may
as well, by symmetry, replace the bounded domain $\mathcal{A}$ with
a strip so that, for $t$ real, $\phi_{t}$ gets identified with a
function on $\Omega\times[0,1].$ In this latter notation there is
a similar construction of a barrier $\chi_{t}$ as in the compact
case, namely \begin{equation}
\chi_{t}:=\max\{\phi_{0}-A\Re t,\phi_{1}-A(1-\Re t),A\rho\}\label{eq:barrier in compact case-1}\end{equation}
where $\rho$ is a psh exhaustion function of $\Omega$ (e..g. $\rho=|z|^{2}-1$
in the ball case). It hence determines an extension $F$ of $f$ such
that $F\in\mathcal{C}^{0}(\bar{M})\cap PSH(M)$ and hence $\Phi$
is bounded on $M$ and converges uniformly towards the right boundary
vales. In fact, given the extension $F$ above it follows from a theorem
in \cite{bl} (since $M$ is hyperconvex) that $\Phi\in\mathcal{C}^{0}(\bar{M})\cap PSH(M)$
with \[
(dd^{c}\Phi)^{n+1}=0,\,\,\,\mbox{in\,\ensuremath{M}}\]
(but strictly speaking we will not need the continuity, only the boundedness
and the uniform boundary behavior as $t\rightarrow0$ and $t\rightarrow1).$
In particular we obtain a continuous curve $\phi_{t}$ in the space
$PSH\cap L^{\infty}(\Omega).$ 
\begin{lem}
\label{lem:prop of energy along geod in domain}Let $\phi_{t}$ be
a geodesic segment as above. 
\begin{itemize}
\item For any fixed $t$ we have that $\phi_{t}\in\mathcal{E}^{1}(\Omega)$
and if\textup{ $\dot{\phi}_{0}$ denotes the right derivative of $\phi_{t}$
at $t=0$ (which exists by convexity), then} \textup{\[
\frac{d}{dt}_{t=0^{+}}\mathcal{E}(\phi_{t})\leq\int_{\mathcal{B}}\dot{\phi}_{0}(dd^{c}\phi_{0})^{n}/n!,\]
}
\item $t\mapsto\mathcal{E}(\phi_{t})$ is affine and continuous on $[0,1].$ 
\end{itemize}
\end{lem}
\begin{proof}
As explained above $\chi_{t}\leq\phi_{t}\leq0$ where $\chi_{t}$
is a maximum of functions in $\mathcal{E}^{1}(\Omega)$ and hence
$\chi_{t}$ is also in the space $\mathcal{E}^{1}(\Omega)$ \cite{ce0}.
By Lemma \ref{lem:first deriv of energy in domain} the functional
$\mathcal{E}$ is increasing on $\mathcal{E}^{1}(\Omega)$ (since
its differential is a positive measure) and hence $-\infty<\mathcal{E}(\chi_{t})\leq\mathcal{E}(\phi_{t}^{j})\leq0$
for any sequence $\phi_{t}^{j}$ in $\mathcal{H}_{0}(\Omega)_{b}$
decreasing to $\phi,$ which proves the first claim. Next, we recall
that $\mathcal{E}$ is concave on $\mathcal{E}^{1}(\Omega)$ (wrt
the usual affine structure) which for example follows from the formula
for $d_{t}d_{t}^{c}\mathcal{E}(\phi_{t})$ discussed below. In particular,
\[
\frac{1}{t}(\mathcal{E}(\phi_{t})-\mathcal{E}(\phi_{0}))\leq\frac{1}{t}\int_{\Omega}(\phi_{t}-\phi)(dd^{c}\phi_{0})^{n}/n!,\]
so that letting $t\rightarrow0^{+}$ proves the first point. As for
the last point integration by parts show that the formula \ref{eq:second deriv of energy}
for $d_{t}d_{t}^{c}\mathcal{E}(\phi_{t})$ is still valid in the smooth
case. However, as we will need the formula in a singular setting we
instead refer to the result proved in \cite{ce} which implies that
if $\Phi\in\mathcal{F}(\mathcal{\Omega\times}\mathcal{A})$ whose
slices $\phi_{t}$ are in $\mathcal{E}^{1}(\Omega)$ then the analogue
of formula \ref{eq:second deriv of energy} holds (i.e. for $X=\Omega)$
in the sense of currents. Finally, since $\phi_{t}\rightarrow\phi_{0}$
uniformly as $t\rightarrow0$ we have that $\mathcal{E}(\phi_{t})\rightarrow\mathcal{E}(\phi_{0})$
as $t\rightarrow0$ \cite{ce0} and similarly for $t\rightarrow1$
and that ends the proof.
\end{proof}

\subsection{The case of the ball and the proof of Theorem \ref{thm:sharp m-t in ball under inv intro}}

We now take $\Omega$ to be the unit-ball and make a special choice
of reference function $\phi_{0}$ as \[
\phi_{0}=(n+1)[\log(1+|z|^{2})-\log2].\]
 This is a potential of the Fubini-Study metric on $\P^{n}$ and satisfies
the Kähler-Einstein equation \begin{equation}
(dd^{c}\phi_{0})^{n}/n!=a_{n}e^{-\phi_{0}}\label{eq:k-e equation in ball}\end{equation}
 We first prove an estimate for the Bergman kernel at the origin.
By Prop \ref{pro:bergman kernel under s one inv} this amounts to
an estimate of the integral of $e^{-\phi}$ in the $S^{1}$-invariant
case, but we prefer to argue first in the general case, since we feel
the estimate for Bergman kernels has independent interest. 
\begin{prop}
\label{pro:m-t for bergman kernel in ball}Let $\phi$ be a smooth
plurisubharmonic function in the ball that vanishes on the boundary.
Then \[
-\log K_{\phi}(0)\leq\log\int_{\mathcal{B}}e^{-\phi_{0}}-b_{n}\mathcal{E}(\phi,\phi_{0})\]
 where $b_{n}=(a_{n}\int e^{-\phi_{0}})^{-1}$. \end{prop}
\begin{proof}
As in the compact setting the proof uses geodesics. We connect $\phi$
and $\phi_{0}$ by a geodesic $\phi_{t}$ such that $\phi_{1}=\phi$.
Then \[
g(t):=\log K_{\phi_{t}}-b_{n}\mathcal{E}(\phi_{t},\phi)\]
 is a convex function of $t$. We claim that $g'(0)\geq0$. For this
we use the same formula as before for the derivative of the Bergman
kernel \[
\frac{d}{dt}K_{\phi_{t}}(x)=\int_{\mathcal{B}}\dot{\phi}_{t}|K_{\phi_{t}}(x,y)|^{2}e^{-\phi_{t}}.\]
 Take $x=0$ and $t=0$. Then, since $\phi_{0}$ is $S^{1}-$symmetric
\[
K_{\phi_{0}}(0,y)=1/\int_{\mathcal{B}}e^{-\phi_{0}}\]
 by proposition \ref{pro:bergman kernel under s one inv}. Therefore
\[
\frac{d}{dt}|_{t=0}\log K_{\phi_{t}}(0)=\int_{\mathcal{B}}\dot{\phi}_{0}e^{-\phi_{0}}/\int_{\mathcal{B}}e^{-\phi_{0}}.\]
 Combining this with the Kähler-Einstein condition \ref{eq:k-e equation in ball}
we get, also using Lemma \ref{lem:prop of energy along geod in domain},
that \[
\frac{d}{dt}|_{t=0}\log K_{\phi_{t}}(0)=b_{n}\int_{X}\dot{\phi}_{0}(dd^{c}\phi_{0})^{n}/n!\geq b_{n}\frac{d}{dt}|_{t=0}\mathcal{E}(\phi_{t},\phi).\]
 so $g'(0)\geq0$ as claimed. Since $g$ is moreover convex we get
$g(1)\geq g(0)$ or explicitly \[
\log K_{\phi}(0)-b_{n}\mathcal{E}(\phi,\phi_{0})\geq\log K_{\phi_{0}}(0).\]
 Invoking \ref{pro:bergman kernel under s one inv} again the proposition
follows. 
\end{proof}
From here we can not continue as in the compact case since we have
no counterpart of \ref{eq:integral of bergman is dim}. It seems plausible
to conjecture that for any compact $K$ in the ball \[
\int_{K}K_{\phi}(z,z)e^{-\phi(z)}\leq C(K,\phi)\]
 where the constant depends only on $K$ and, say, \[
\int_{\B}(dd^{c}(\phi+|z|^{2}))^{n}.\]
 If this were true we could follow a route similar to what we did
in the case of a compact manifold and obtain sharp estimates for \[
\int_{K}e^{-\phi}\]
 for functions that are not necessarily $S^{1}$-invariant. The most
one could hope for in this direction would be \[
\int_{\B}(1-|z|^{2})^{n+1}K_{\phi}(z,z)e^{-\phi(z)}\leq C(\phi)\]
 with the same dependence on $\phi$. We do not know if either of
these estimates hold.

Instead we now introduce the additional assumption that $\phi$ be
$S^{1}$-invariant. We then get, by Proposition \ref{pro:bergman kernel under s one inv},
that \[
\log\int_{\mathcal{B}}e^{-\phi}\leq\log\int_{\mathcal{B}}e^{-\phi_{0}}-b_{n}\mathcal{E}(\phi,\phi_{0})\]
 if $\phi$ is any smooth plurisubharmonic function in the ball, vanishing
on the boundary and $S^{1}$-invariant.

\bigskip{}

As it stands the constant here is not optimal. An easy way to improve
it is to replace our 'reference' $\phi_{0}$ by \begin{equation}
\phi_{0}^{\epsilon}:=(n+1)[\log(\epsilon^{2}+|z|^{2})-\log(\epsilon^{2}+1)].\label{eq:rescaled f-s in unit-ball}\end{equation}
 This amounts to replacing the unit ball by a larger ball of radius
$1/\epsilon$ which brings us closer and closer to all of $\P^{n}$,
where the same argument is known to give an optimal constant. Then
\[
(dd^{c}\phi_{\epsilon})^{n}/n!=a_{n}(\epsilon)e^{-\phi_{\epsilon}},\]
 and as before we let \[
b_{n}(\epsilon)=1/(a_{n}(\epsilon)\int e^{-\phi_{\epsilon}}).\]
 By the Kähler-Einstein equation for $\phi_{\epsilon}$ \[
b_{n}(\epsilon)=n!/\int(dd^{c}\phi_{\epsilon})^{n}.\]
 The integral here is easily computed using Stokes' theorem \[
\int_{|z|<1}(dd^{c}\phi_{\epsilon})^{n}=\int_{|z|=1}d^{c}\phi_{\epsilon}\wedge(dd^{c}\phi_{\epsilon})^{n-1}=\]
 \[
=(n+1)^{n}(1+\epsilon^{2})^{-n}\int_{|z|=1}d^{c}|z|^{2}\wedge(dd^{c}|z|^{2})^{n-1}=(n+1)^{n}(1+\epsilon^{2})^{-n}\int_{|z|<1}(dd^{c}|z|^{2})^{n}=\]
 \[
=(n+1)^{n}(1+\epsilon^{2})^{-n}n!\pi^{-n}|\B_{n}|=(n+1)^{n}(1+\epsilon^{2})^{-n}.\]
 Hence $b_{n}(\epsilon)$ is asymptotic to $n!/(n+1)^{n}$ as $\epsilon$
goes to zero (coinciding with the inverse of the volume of $\P^{n},$
as it must). We have \[
\mathcal{E}(\phi_{\epsilon})=(n+1)^{-1}\int_{\mathcal{B}}\phi_{\epsilon}(dd^{c}\phi_{\epsilon})^{n}/n!\]
 which by the Kähler-Einstein equation equals \[
-a_{n}(\epsilon)\int_{\mathcal{B}}\log(\epsilon^{2}+|z|^{2})e^{-\phi_{\epsilon}}\]
 plus a quantity tending to zero with $\epsilon$. Thus \[
b_{n}\mathcal{E}(\phi_{\epsilon})=-\int_{\mathcal{B}}\log(\epsilon^{2}+|z|^{2})e^{-\phi_{\epsilon}}/\int_{\mathcal{B}}e^{-\phi_{\epsilon}}.\]
 This is the integral of $-\log(\epsilon^{2}+|z|^{2})$ against a
sequence of measures that tend to a Dirac unit mass at the origin,
and it is easily seen to be asymptotic to a constant plus $-\log\epsilon^{2}$.On
the other hand \[
-\log\int_{\mathcal{B}}e^{-\phi_{\epsilon}}\]
 is also asymptotic to $-\log\epsilon^{2}$ plus a constant. All in
all this proves Theorem \ref{thm:sharp m-t in ball under inv intro}
stated in the introduction. 

Notice that there seems to be no extremal function for the inequality.
For any nonzero $\epsilon$, $\phi_{0}^{\epsilon}$ is an extremal
by construction, but these functions tend to $(n+1)\log|z|^{2}$,
which has infinite energy.

We do not know if \ref{thm:sharp m-t in ball under inv intro} holds
without our assumption of $S^{1}-$symmetry except for $n=1$, see
\cite{m} where a symmetrization argument can be used. Our methods
also have bearings on symmetrization properties in the present higher
dimensional setting of domains in $\C^{n}$ and we hope to come back
to this point in the future. In section \ref{sec:Moser-Trudinger-and-Demailly}
we shall use a different argument to prove the inequality 'modulo
$\epsilon$' without assuming $S^{1}$-invariance.

\section{\label{sec:Moser-Trudinger-and-Demailly}Moser-Trudinger and Brezis-Merle
type inequalities on domains in $\C^{n}$}

Let $\Omega$ be a bounded hyperconvex domain in $\C^{n}.$ We may
then set the reference form $\omega_{0}$ to be the zero-form: $\omega_{0}=0$
and use the notation \[
\mathcal{E}(u):=\mathcal{E}_{\omega_{0}}(u)=\frac{1}{(n+1)!}\int_{\Omega}u(dd^{c}u)^{n}\]
 It will also be convenient to write \[
\mathcal{M}(u):=\int_{\Omega}(dd^{c}u)^{n}\]

We will say that the\emph{ sharp Moser-Trudinger (M-T) inequality
}holds for the domain $\Omega$ if there is a constant $C$ such that
\[
(\mbox{M-T})\,\,\,\,\,\log\int_{\Omega}e^{-u}dV\leq-\frac{n!}{(n+1)^{n}}\mathcal{E}(u)+C\]
for any $u\in\mathcal{E}_{1}(\Omega).$ Similarly, the\emph{ quasi-sharp
M-T inequality} is said to hold on $\Omega$ if for any $\delta>0$
the previous inequality holds when the factor $n+1$ in front of $\mathcal{E}(u)$
is replaced by $n+1-\delta$ and the constant $C$ by $C-\log(\delta^{(n-1)}).$ 

The \emph{sharp Brezis-Merle (B-M) inequality} is said to hold for
the domain $\Omega$ if there is a constant $A$ such that \[
(\mbox{B-M}):\,\,\,\,\int_{\Omega}e^{-u}dV\leq A\left(1-\frac{1}{n^{n}}\mathcal{M}(u)\right)^{-1}\]
for any $u\in\mathcal{F}(\Omega)$ such that $\mathcal{M}(u):=\int_{\Omega}(dd^{c}u)^{n}<n^{n}.$ 

It will also be convenient to used the following equivalent formulations
of the quasi-sharp Moser-Trudinger and Brezis-Merle inequalities:
\[
(\mbox{M-T}')\,\,\,\,\,\int_{\Omega}e^{-(n+1-\delta)u}dV\leq C\delta^{-(n-1)}e^{-(n+1-\delta)n!\mathcal{E}(u)}\]
for some positive constant $C$ and (when $n>1):$ there is a positive
constant $A$ such that \begin{equation}
(\mbox{B-M}'):\,\,\,\,\,\int_{\Omega}e^{-(n-\delta)u}dV\leq A\delta^{-(n-1)}\label{eq:alt def of quasi sharp b-m}\end{equation}
for all $u\in\mathcal{F}(\Omega)$ such that $\mathcal{M}(u)=1$

\subsection{M-T in $\C^{n}$ implies B-M in $\C^{n+1}$}
\begin{prop}
\label{pro: m-t impl demailly }The (quasi-) sharp Moser-Trudinger
inequality on $\Omega\subset\C^{n}$ implies the (quasi-) sharp Brezis-Merle
inequality\emph{ }on $\Omega\times D\subset\C^{n+1}.$ More generally,
the (quasi-) sharp Moser-Trudinger inequality on the ball in $\C^{n}$
implies the (quasi-) sharp Brezis-Merle inequality on any hyperconvex
domain in $\C^{n+1}.$\end{prop}
\begin{proof}
Let us start with the sharp case. Given $u\in\mathcal{F}(\Omega_{z}\times D_{t})$
we let $v(t):=\mathcal{E}(u(t,\cdot)$ and to fix ideas we first assume
that $u$ is smooth on the closure of $\Omega\times D.$ Applying
the sharp M-T inequality to $u(t,\cdot)$ for $t$ fixed and integrating
over $t\in D$ gives \[
\int_{D}(\int_{\Omega}e^{-u(t,z)}dV(z))dV(t)\leq\int_{D}\exp(-\frac{n!}{(n+1)^{n}}v(t)dV(t),\]
By \ref{eq:second deriv of energy} the function $v(t)$ is a subharmonic
function on $D$ with $\int_{D}d_{t}d_{t}^{c}v=\int_{\Omega\times D}(dd^{c}u)^{n+1}/(n+1)!.$
Hence applying the sharp B-M inequality on the disc $D$ for $n=1$
(which is follows from Green's formula and Jensen's inequality \cite{br-m}
or alternatively from Polya's inequality \cite{ce}) and using that
$\frac{n!}{(n+1)^{n}}\frac{1}{(n+1)!}=\frac{1}{(n+1)^{n+1}}$ finishes
the proof under the smoothness assumption above. The general case
if proved in a similar way, but using the singular variant of \ref{eq:second deriv of energy}
proved in \cite{ce} (Theorem 3.1); compare the proof of Lemma \ref{lem:prop of energy along geod in domain}.
To prove the last statement we recall the \emph{subextension theorem}
\cite{c-z} saying that given $\Omega$ and $\tilde{\Omega}$ two
hyperconvex domains such that $\Omega\subset\tilde{\Omega}$ and a
function $u\in\mathcal{F}(\Omega)$ there is a function $\tilde{u}\in\mathcal{F}(\tilde{\Omega})$
such that $\tilde{u}\leq u$ on $\Omega$ and $\int_{\tilde{\Omega}}MA(\tilde{u})\leq\int_{\Omega}MA(u)$
(up to taking approximations $\tilde{u}$ is obtained by solving the
Dirichlet problem $MA(\tilde{u})=1_{\Omega}MA(u)$ on $\tilde{\Omega}).$
Applying subextension to $\Omega\subset r(\mathcal{B}\times D)$ for
$r$ sufficiently large thus shows that the sharp B-M inequality holds
on any hyperconvex domain $\Omega.$ Finally, if we instead assume
that the quasi-sharp M-T holds in dimension $n-1$ and take $u$ such
that $\mathcal{M}(u)=1$ then repeating the same argument gives, with
$v=(n+1-\delta)\mathcal{E}(u(t,\cdot)$ that 

\[
\int_{D}(\int_{\Omega}e^{-(n+1-\delta)u(t,z)}dV(z))dV(t)\leq C'\delta^{-n}\left(1-\frac{(n+1-\delta)^{n}}{(n+1)^{n}}\right)^{-1}\]
and expanding $1-t^{n}=(1-t)(1+...+t^{n})$ then concludes the proof.
\end{proof}

\subsection{\label{sub:Quasi-B-M-in}Quasi B-M in $\C^{n}$ implies quasi M-T
in $\C^{n}$ and the free energy functional}

In this section it will be convenient to use a different normaliation
of $\mathcal{E}$ obtained by multiplication by $n!,$ i.e. we let
\[
\mathcal{E}(u):=\frac{1}{n+1}\left\langle u,(dd^{c}u)^{n}\right\rangle ,\,\,\,\,\left\langle u,\mu\right\rangle :=\int_{\Omega}u\mu\]
With this new normalization $d\mathcal{E}_{|u}=(dd^{c}u)^{n}$ and
the sharp M-T inequality may be formulated as $\int e^{-(n+1)u}\leq Ce^{-(n+1)\mathcal{E}(u)}.$ 
\begin{prop}
\label{pro:Demailly imples m-t}If the quasi-sharp Brezis-Merle inequality
holds on $\Omega\subset\C^{n}$ than so does the quasi-sharp Moser-Trudinger
inequality.
\end{prop}
The proof uses the {}``thermodynamical formalism'' recently introduced
in a the setting of compact Kähler manifolds in \cite{berm2}. The
key point is to show that, by Legendre duality, the (sharp) Moser-Trudinger
inequality is equivalent to yet another inequality, namely one which
coincides with the classical \emph{logarithmic Hardy-Sobolev (LHS)
inequality} when $n=1.$ To make this precise we first define, for
any given positive number $\gamma,$ \[
\mathcal{G}_{\gamma}(u):=\mathcal{E}(u)-\mathcal{L_{\gamma}}(u),\,\,\,\,\mathcal{L_{\gamma}}(u)=-\frac{1}{\gamma}\log\int_{X}e^{-\gamma u}dV,\]
 where $u\in\mathcal{E}^{1}(\Omega)$ so that $\mathcal{G}_{\gamma}$
is bounded from above for $\gamma=n+1$ precisely when the sharp Moser-Trudinger
inequality holds. As for the LHS type inequality referred to above
it is said to hold when the following\emph{ free energy functional}
$F_{\gamma}$ is bounded from above:

\[
F_{\gamma}(\mu):=E(\mu)-\frac{1}{\gamma}D(\mu)\]
 where $\mu$ is a probability measure on $\bar{\Omega}$ with $E(\mu)<\infty,$
where $E(\mu)$ is the (pluricomplex) energy of $\mu$ and $D(\mu)$
is its relative entropy, whose definitions we next recall. Following
\cite{ce0} a measure $\mu$ on $\Omega$ is said to have finite\emph{
(pluricomplex) energy} $E(\mu)$ if it admits a finite energy potential
$u_{\mu},$ i.e. $u_{\mu}\in\mathcal{E}^{1}(\Omega)$ and \begin{equation}
(dd^{c}u_{\mu})^{n}=\mu\label{eq:inhomo ma eq in domain}\end{equation}
One may then define its energy by \[
E(\mu):=-\frac{n}{n+1}\left\langle u_{\mu},\mu\right\rangle \]
 which is finite and non-negative (the reason for our normalization
appears in formula \ref{eq:variational def of energy} below). If
$u_{\mu}$ does not exist one sets $E(\mu)=\infty.$ We also recall
the classical notion of relative entropy: given a measure $\mu$ its
\emph{relative entropy} (wrt $dV)$ is defined as \[
D(\mu):=\int\log(\mu/dV)\mu\]
if $\mu$ is a probability measure which is absolutely continuous
wrt $dV$ (with density $\mu/dV)$ and otherwise $D(\mu):=\infty.$
To see the relation to the Moser-Trudinger inequality we recall that
$E$ and $\frac{1}{\gamma}D$ can be realized as Legendre type transforms
of the concave functional $\mathcal{E}$ and $\mathcal{L_{\gamma}},$
respectively. Indeed, it is a classical fact (see \cite{berm2} and
references therein) that \begin{equation}
\frac{1}{\gamma}D(\mu)=\mathcal{L_{\gamma}}^{*}(\mu):=\sup_{u\in C^{0}(X)}\left(-\frac{1}{\gamma}\log\int_{X}e^{-\gamma u}\mu_{0}-\left\langle u,\mu\right\rangle \right)\label{eq:entropy as legendre tr}\end{equation}
Moreover, it follows from the concavity of $\mathcal{E}$ and the
solvability of equation \ref{eq:inhomo ma eq in domain} that %
\footnote{In fact, using a variational approach the potential $u_{\mu}$ above
may be obtained directly by maximizing the functional in the rhs of
\ref{eq:variational def of energy}. This was recently shown in the
Kähler setting in \cite{bbgz} and in the setting of domains in \cite{a-c-c}
(compare section \ref{sec:Existence-of-optimizers}). %
} \begin{equation}
E(\mu)=\sup_{u\in\mathcal{E}^{1}(\Omega)}\mathcal{E}(u)-\left\langle u,\mu\right\rangle \label{eq:variational def of energy}\end{equation}
The idea is now to first show that \begin{equation}
F_{\gamma}\leq C_{\gamma}\implies\mathcal{G}_{\gamma}:=\mathcal{E}-\mathcal{L_{\gamma}}\leq C_{\gamma}\label{eq:bound on F implies bound on G}\end{equation}
 and then prove that $F_{\gamma}\leq C_{\gamma}$ for $\gamma<n+1,$
giving the desired M-T inequality. If $E$ were a proper Legendre
transform of $\mathcal{E}$ (i.e. if the sup in \ref{eq:variational def of energy}
could be taken over $C^{0}(\bar{\Omega})$ then \ref{eq:bound on F implies bound on G}
would follow immediately from the fact that the Legendre transform
is involutive together with the trivial implication \[
f\leq g+C\implies f^{*}\leq g^{*}+C\left(\implies f\leq g+C\right)\]
In the Kähler setting it was explained how to use a certain projection
operator $P$ to realize $E$ the Legendre transform of $\mathcal{E\circ}P$,
but for the implication \ref{eq:bound on F implies bound on G} this
will not be needed. Indeed, by the concavity of $\mathcal{E}$ on
$\mathcal{E}^{1}(\Omega)$ we have, for any fixed measure $\mu,$
\begin{equation}
\mathcal{E}(u)\leq\mathcal{E}(u_{\mu})+\left\langle u-u_{\mu},\mu\right\rangle =E(\mu)+\left\langle u,\mu\right\rangle =F_{\gamma}(\mu)+\left(\frac{1}{\gamma}D(\mu)+\left\langle u,\mu\right\rangle \right)\label{eq:pf of f bounded implies}\end{equation}
The proof may now be concluded by noting that (compare \ref{eq:entropy as legendre tr})
\[
\inf_{\mu}(\frac{1}{\gamma}D(\mu)+\left\langle u,\mu\right\rangle )=\mathcal{L_{\gamma}}(u)\]
where the infimum is taken over all measures on $\Omega.$ More concretely,
we may by approximation, assume that $u\in\mathcal{H}_{0}(\Omega)$
and then note that $\mu=e^{-\gamma u}/\int e^{-\gamma u}dV$ realizes
the inf above in \ref{eq:pf of f bounded implies}, so that the previous
argument gives \[
\mathcal{G}_{\gamma}(u)\leq F_{\mu}(e^{-\gamma u}/\int e^{-\gamma u}dV)\leq C_{\gamma}\]
proving \ref{eq:bound on F implies bound on G}.

Finally, to estimate $F_{\gamma}$ we next define the following general
invariant of a pair $(\Omega,\mu_{0})$ where $\mu_{0}$ is a measure
on $\Omega:$ \begin{equation}
\alpha(\Omega,\mu_{0}):=\sup\left\{ t:\,\exists C_{t};\,\int_{\Omega}e^{-tu}d\mu_{0}\leq C_{t}\,\,\forall u\in\mathcal{H}_{0}(\Omega)_{b}\cap\{\int_{\Omega}(dd^{c}u)^{n}/n!=1\right\} \label{eq:alpha invariant for domain}\end{equation}

\begin{lem}
\label{lem:bounded using alpha}If $\gamma<\alpha\frac{(n+1)}{n},$
then $F_{\gamma}(\mu)$ is bounded from above, i.e. $F_{\gamma}(\mu)\leq C_{\gamma}.$
More precisely, for any $t<\alpha$ \[
F_{\gamma}(\mu)\leq(\frac{t}{\gamma}-\frac{n}{n+1})\left\langle u_{\mu},\mu\right\rangle +\frac{t}{\gamma}C_{t}\]
where $C_{t}$ is the minimum of $\mathcal{L}_{t}(u)$ over all $u\in\mathcal{H}_{0}(\Omega)_{b}\cap\{\int_{\Omega}(dd^{c}u)^{n}=1.$ \end{lem}
\begin{proof}
Given $\gamma$ we fix $t<\alpha:=\alpha(\Omega,\mu_{0}).$ By the
definition of $\alpha$ we have $\mathcal{L}_{t}(u)\geq-C_{t}$ if
$u\in\mathcal{H}_{0}(\Omega)\cap\{\int_{\Omega}(dd^{c}u)^{n}=1$ and
hence \[
\frac{1}{t}D(\mu)=\mathcal{L}_{t}^{*}(\mu)\geq\mathcal{L}_{t}(u_{\mu})-\left\langle u_{\mu},\mu\right\rangle \geq-\left\langle u,\mu_{\mu}\right\rangle -C_{t}\]
As a consequence \[
F_{\gamma}(\mu)\leq(-\frac{n}{n+1}+\frac{t}{\gamma})\left\langle u_{\mu},\mu\right\rangle +tC_{t}\]
Given $\gamma$ such that $\gamma<\alpha\frac{(n+1)}{n},$ we may
now choose $t$ sufficiently close to $\alpha$ so that the multiplicative
constant above is strictly positive, thus concluding the proof.
\end{proof}
Assume now that the quasi-sharp BM-inequality holds in $\Omega.$
The point is that this implies that $\alpha(\Omega,dV)=n$ and the
previous Lemma then shows that $F_{n+1-\delta}$ is bounded from above.
However, we can actually be more precise wrt the depends on $\delta.$
Indeed, according to the formulation \ref{eq:alt def of quasi sharp b-m}
we have that $C_{n-\epsilon}\leq C+\log(1/\epsilon^{n-1})$ where
$C_{t}$ is defined as in the previous lemma (with $\mu_{0}=dV).$
Applying the previous lemma with $\gamma=n+1-\delta$ and $t=n-\delta/2$
hence gives \[
F_{n+1-\delta}(\mu)\leq C_{n-\delta/2}\leq C'+\log(1/\delta^{n-1})\]
 The proof of Prop \ref{pro:Demailly imples m-t} is now concluded
by using \ref{eq:bound on F implies bound on G}. 
\begin{rem}
When $\mu_{0}=dV$ is any volume form on $\bar{\Omega}$ $\alpha:=\alpha(\Omega,dV)$
defines an invariant of a domain $\Omega$ which can be seen as a
variant of Tian's $\alpha-$invariant for a Kähler manifold $(X,\omega)$
(or rather the class $[\omega]).$ The difference is that in the latter
case the Monge-Ampère mass is, of course, determined by $[\omega]$
and hence independent of $u.$ In this letter setting $-\gamma F_{\mu}(MA(u))$
coincides with Mabuchi's K-energy functional, which plays a key role
in Kähler geometry (compare the discussion in \cite{berm2})
\end{rem}

\subsection{Proof of Theorem \ref{thm:quasi-sharp m-t in ball intro} }

The sharp Moser-Trudinger inequality holds when $n=1$ in the disc
$D$ \cite{m}. Hence combining Prop \ref{pro: m-t impl demailly }
and Prop \ref{pro: m-t impl demailly } simultaneously prove the inequalities
in Theorem \ref{thm:quasi-sharp m-t in ball intro} and Theorem \ref{thm:quasi-sharp b-m intro}. 

As for the sharpness of the multiplicative constants in inequalities
we make the following remark which concludes the proof of Theorem
\ref{thm:quasi-sharp m-t in ball intro}...
\begin{rem}
\label{rem:no better constants}Let $\Omega:=\mathcal{B}$ be the
unit-ball in $\C^{n}$ and set $u:=\log|z|^{2}$ so that $(dd^{c}u)^{n}=\delta_{0}.$
Letting $u_{t}:=tu$ for $t<1$ gives $\int_{B}e^{-u_{t}}=\frac{1}{1-t/n}\sim\frac{1}{(1-\frac{t^{n}}{n^{n}})}$
as $t\rightarrow1^{-}.$ Moreover, since $MA(u_{t})=t^{n}$ this shows
that the sharp Brezis-Merle inequality cannot hold on $\mathcal{B}$
with a better coefficient than $\frac{1}{n^{n}},$ nor with a smaller
power in the rhs. Using the subextenstion theorem (see the proof of
Theorem ) gives the same conclusion for any hyperconvex domain $\Omega$
(alternatively when can apply the same argument with $u$ replaced
by the pluricomplex Green function $g_{z}$ with a pole at any fixed
point $z$ in $\Omega).$ Finally, by Prop \ref{pro: m-t impl demailly }
this also shows that the coefficient $n!/(n+1)^{n}$ in the sharp
M-T inequality cannot be improved for any hyperconvex domain $\Omega.$ 
\end{rem}

\section{Relations between the various inequalities }

Let $u\leq0$ be a, say continuous, function on a topological space
$X$ and $dV$ a finite measure on $X.$ We let \[
E(t):=\int e^{-tu}dV\]
and \[
V(s):=\mbox{Vol }\{u<-s\}:=\int_{\{u<-s\}}dV\]
Then $E(t)/t$ and $V(s)$ are (up to signs) related by Laplace transforms.
Indeed, by the push-forward formula and integration by parts \[
E(t):=t\int_{0}^{\infty}e^{ts}V(s)\]
According to a well-known principle the Laplace transform is asymptotically
described by the Legendre transform: \[
E(t)\lesssim e^{f(t)}\,\,\,"\iff"\,\,\, V(s)\lesssim e^{-f^{*}(s)}\]
(as $t$ and $s$ tend to infinity), where $f$ is assumed convex
and $f^{*}(s)$ is its Legendre transform: \[
f^{*}(s):=\sup_{t}(st-f(t))\]
There are various ways of formulating this principle precisely but
for our purposes the following basic lemma will be sufficient:
\begin{lem}
\label{lem: laplace est}If $E(t)\leq Ce^{f(t)},$ then $V(t)\leq Ce^{-f^{*}(s)}.$
Conversely, if $V(t)\leq Ce^{-g(s)}$ then for any $\delta>0$ there
is a constant $C_{\delta}$ such that $E(t)\leq C_{\delta}e^{g^{*}(t+\delta)}.$ \end{lem}
\begin{proof}
Fix $t\in\R.$ On the subset $\{u<-s\}$ of $X$ we have $1<e^{-st}e^{-tu}$
and hence $V(s)\leq e^{-st}\int_{X}e^{-tu}\leq Ce^{-st+f(t)}.$ Taking
the infimum over all $t$ then proves the first inequality. The second
inequality follows immediately from the definitions if we rewrite
$ts-g(s)=\left((t+\delta)s-g(s)\right)-\delta s$ and let $C_{\delta}=C\int_{0}^{\infty}e^{-\delta s}ds=C/\delta.$
\end{proof}
We will apply the previous lemma to the case when $f(t)$ is homogeneous
and use the following basic relations (assuming $p>1)$ \begin{equation}
f(t)=\frac{1}{a}s^{p}/p\iff f^{*}(s):=a^{(q-1)}t^{q}/q\label{eq:legendre transforms for homo}\end{equation}
where $1/p+1/q=1$ (the case $a=1$ is immediate and implies the general
case by scaling). More precisely, in our case we will have $p=(n+1)/n$
and hence $q=n+1$ and vice versa. 
\begin{cor}
(of Theorem \ref{thm:aubin hyp intro}): Let $(X,\omega)$ be a compact
Kähler manifold and $u\in\mathcal{H}_{0}(X,\omega).$ Then there are
constants $A$ and $B$ such that \[
\mbox{Vol }_{\omega}\{u<-s\}\leq Ce^{-B\frac{1}{(-\mathcal{E}_{\omega}(u))^{1/n}}s^{(n+1)/n}}\]
 More precisely, we may replace the exponent above by \[
-\frac{n}{(-\mathcal{E}_{\omega}(u))^{1/n}(n+1)^{(1+1/n)}}s^{(n+1)/n}(1+o(1))\]
 as $s\rightarrow\infty.$
\end{cor}
From the first volume estimate in the previous corollary we see that
the $L^{p}-$norms of $u$ may be estimated as \[
\int_{X}(-u)^{p}dV=\int_{0}^{\infty}V(s)d(s^{p})\leq C\Gamma(\frac{n}{n+1}p)\left(\frac{1}{B}\right)^{pn/(n+1)}(-\mathcal{E}_{\omega}(u))^{p/(n+1)}\]
(after setting $x=s^{(n+1)/n}$ and using $\Gamma(x)x=\Gamma(x+1),$
for $\Gamma(x):=\int_{0}^{\infty}s^{x-1}e^{-s}ds).$ Using $\Gamma(m)=(m-1)!$
and Stirling's approximation $m!\sim(m/e)^{m}$ hence gives the Sobolev
type inequality \ref{eq:sob intro} from the introduction. 

The inequality \ref{eq:trudingers ineq on kahler} can now be deduced
from the previous Sobolev type inequality (compare \cite{tr}). Indeed,
assuming first that $-\mathcal{E}_{\omega}(u)=1$ gives \[
\int e^{B(1-\delta)(-u)^{n+1/n}}dV=\sum_{p=1}^{\infty}\frac{B^{j}}{j!}\int_{X}(-u)^{j(n+1)/n}dV\leq\sum_{p\in\N(n+1)/n}\frac{1}{p}(1-\delta)^{pn/(n+1)}\]
which is finite for any $\delta>0$ and the general case then follows
by scaling. Note in particular, that when $E(t)\leq e^{At^{n+1}}$
with $A=(n+1)^{-(n+1)}$ then $V(s)\leq e^{-Bs^{(n+1)/n}}$ with $B=n$
which proves the last statement in Theorem \ref{thm:quasi-sharp m-t in ball intro}.

\section{\label{sec:Remarks-on-the}Remarks on the optimal constants }

In this section we will compare our results with Aubin's conjectures
\cite{au,au2} (and partial results). To this end we first have to
compare our notations, which differ slightly. There are two reasons
for the differences which come from (1) the choice of energy functional
(2) the normalizations of the energy functional. We start with the
energy functionals (in our normalizations). Given the functional $\mathcal{E}_{\omega}$
which we recall may be defined as a primitive of the Monge-Ampère
operator one defines \[
J_{\omega}(u):=-\mathcal{E}_{\omega}(u)+\int u\omega^{n}/n!\]
and \[
I_{\omega}(u):=\frac{1}{n!}\int(-u)\left(\omega_{u}{}^{n}-\omega^{n}\right)\]
In particular, the functionals $J_{\omega}$ and $I_{\omega}$ are
both $\R-$invariant and semi-positive \cite{au} and when $\omega=0$
(as in the $\C^{n}-$setting) they coincide. In general, they are
equivalent up to multiplicative factors \cite{au}: \[
J_{\omega}\leq I_{\omega}\leq(n+1)J_{\omega}\]
However, Aubin's normalizations are slightly different and obtained
by replacing the factor $1/n!$ above by $(2\pi)^{n}/(n-1)!.$ In
particular, \[
(-\mathcal{E_{\omega}})=d_{n}J_{\omega}^{(A)},\,\,\, d_{n}:=\frac{1}{n}\frac{1}{(2\pi)^{n}}\]
if $\int u\omega=0,$ where the super script $A$ refers to Aubin's
normalizations. In this notation Aubin's general {}``Hypothèse fondamentale''
as formulated in \cite{au} asserts that there exist positive constants
$\xi$ and $C$ such that \begin{equation}
\int e^{-ku}dV\leq C\exp(\xi k^{n+1}I_{\omega}^{(A)}(u))\label{eq:aubins general hyp with his norm}\end{equation}
for all $u\in\mathcal{H}(X,\omega)$ normalized such that $\int_{X}u\omega^{n}=0.$
To see that Theorem \ref{thm:aubin hyp intro} confirms this conjecture
(in the case when $[\omega]$ is an integral class) we recall that
there is a constant $C'$ such that \begin{equation}
\sup u\leq\frac{1}{V}\int u\omega^{n}+C'\label{eq:sup estimate for psh}\end{equation}
 and hence \ref{eq:aubins general hyp with his norm} applied to $u-\sup u$
gives \[
\log\int e^{-ku}dV\leq Ak^{n+1}\left(J_{\omega}(u)+\int(-u)\frac{\omega^{n}}{n!}\right)+(AC'k^{n+1}+B)\]
Thus \ref{eq:aubins general hyp with his norm} holds with $\xi=Ak^{n+1}$
and $C=C_{k}:=\exp(AC'k^{n+1}).$ This means that the constant $C_{k}$
depends on $k$ while Aubin's hypothesis, strictly speaking, says
that it should be independent of $k.$ Anyway, in applications to
existence problems for PDEs the precise value of $C_{k}$ is immaterial
(compare section \ref{sec:Existence-of-optimizers}).

\subsection{Counter-example to Aubin's explicit conjecture in the Fano case}

In his paper Aubin also conjectured that in the Fano setting (with
$[\omega]=c_{1}(-K_{X}))$ the infimum $\xi_{n}$ over all constants
$\xi$ satisfying \ref{eq:aubins general hyp with his norm} for some
$C_{\xi}$ is explicitly given by \begin{equation}
\xi_{n}=\pi^{-n}(n-1)!n^{n}(n+1)^{-(2n+1)}=\pi^{-n}(n-1)!\left(1+\frac{1}{n}\right)^{-n}(n+1)^{-(n+1)}\label{eq:def of aubins constant}\end{equation}
However, there is a simple counter-example to this hypothesis. To
see this we first recall a result of Ding (\cite{din}, Prop 6) which
in our notation may be formulated as follows: if one replaces $I_{\omega}$
in \ref{eq:aubins general hyp with his norm} by $(-\mathcal{E_{\omega}})$
then the corresponding optimal constant $\eta(X)$ satisfies (when
specializing to the case $k=1)$ \[
\eta(X)\geq1/V(X)\]
if $X$ admits non-trivial holomorphic vector fields, i.e. if $H^{0}(TX)\neq\{0\}.$
To get a contradiction it will hence be enough to exhibit a Fano manifold
$X_{n}$ with $H^{0}(TX_{n})\neq\{0\}$ such that \begin{equation}
V(X_{n})<\frac{1}{(n+1)}\frac{d_{n}}{\xi_{n}}\left(=\frac{1}{(n+1)!}\frac{1}{2^{n}}\left(1+\frac{1}{n}\right)^{n}(n+1)^{(n+1)}\right)\label{eq:ineq in counter}\end{equation}
 To this end we may simply set $X_{n}=(\P^{1})^{n}$ so that $c_{1}(X_{n})^{n}=2^{n}.$
Since $V(X_{n}):=c_{1}(X_{n})^{n}/n!$ the previous inequality indeed
holds for all sufficiently large $n.$ Indeed, when dividing out $n!$
the lhs in \ref{eq:ineq in counter} is equal to $2^{n}$ while the
rhs is of the order $(n/2)^{n}.$

\subsection{Comparison with Aubin's constant for the ball}

Let us now turn to the setting of the unit-ball, where $\omega=0$
and consider the corresponding functional $I_{0}^{(A)}$ (called $\mathcal{J}$
with the same normalizations in \cite{au2}), i.e. \[
I_{0}^{(A)}(u):=\frac{1}{(n-1)!}\int(-u)\left((i\partial\bar{\partial}u\right)^{n}\]
 In the case of a radial psh function $u$ in the ball Aubin showed
\cite{au2} that \begin{equation}
\log\int e^{-u}dV\leq a_{n}I_{0}^{(A)}(u)+C,\,\,\, a_{n}=2n^{n}(n+1)^{-(2n+1)}\sigma_{2n-1}^{-1},\label{eq:moser-trudinger in ball in aubins notation}\end{equation}
 where $\sigma_{p}$ denotes the volume of the unit $p-$sphere, giving
$a_{n}=\xi_{n}$ (formula \ref{eq:def of aubins constant}). But by
Theorem \ref{thm:quasi-sharp m-t in ball intro} the \emph{optimal}
constant $c_{n}$ in the equality \ref{eq:moser-trudinger in ball in aubins notation}
is equal to \[
c_{n}=\frac{1}{(2\pi)^{n}}\frac{(n-1)!}{(n+1)!}\frac{n!}{(n+1)^{n}}=\frac{(n-1)!}{(2\pi)^{n}}\frac{1}{(n+1)^{n+1}}\]
Hence, \[
c_{n}=\left(\frac{1}{2}(1+1/n)\right)^{n}a_{n},\]
so that $a_{n}\geq c_{n}$ with equality iff $n=1.$ Accordingly,
Aubin's constant $a_{n}$ is not optimal for $n>1.$

\subsection{Discussion}

It is natural to ask why Aubin expected that the particular value
in formula \ref{eq:def of aubins constant} gives the optimal constant
in the Fano case? We can only speculate on this. But it seems that
Aubin was expecting that the optimal constant in the Fano case coincides
with the optimal constant in the setting of the ball. In fact, in
section 3 in \cite{au} Aubin claims that he has proved that the optimal
constant in the setting of the ball is indeed given by formula \ref{eq:def of aubins constant}.
But as explained in the previous section this is\emph{ not} the optimal
constant in the ball unless $n=1$ (and moreover Aubin only proved
his inequality in the radial case). In particular, it is not the case
that the optimal constant in the Fano case coincides with the case
of the ball (by the counter-example above).

It may be illuminating to give an informal description of counter-examples
to Aubin's expectations (which has the virtue of avoiding comparing
various normalizations). As our arguments used in the proof of Theorem
\ref{thm:sharp m-t in ball under inv intro} show, the optimal constant
in the setting of the ball coincides with the optimal constant in
the Fano setting for $X=\P^{n}.$ But by Ding's result the optimal
constant $\xi(X)$ on any Fano manifold $X$ with $H^{0}(TX)\neq\{0\}$
satisfies \[
\xi(X)\geq C_{n}/V(X)\]
where $C_{n}$ is a universal constant (depending on the particular
normalizations of the energy functionals) with equality if $X$ moreover
admits a Kähler-Einstein metric (by \cite{d-t}). Hence, if the optimal
constant coincided with the one in the setting of the ball, then this
would force \[
V(X)\geq V(\P^{n})\]
 for any Fano $X$ such that $H^{0}(TX)\neq\{0\}.$ But this latter
inequality is clearly violated by $X=(\P^{1})^{n}$ and in fact by
many other $X.$ For example, according to two well-known conjectures
$\P^{n}$ is the unique maximizer of the volume functional among (1)
all Fano $n-$folds with Picard number equal to one (see \cite{h}
for a proof when $n\leq4)$ and (2) among all toric Kähler-Einstein
Fano manifolds (such as $(\P^{1})^{n}).$

\section{\label{sec:Existence-of-optimizers}Existence of extremals and applications
to Monge-Ampère equations}

\subsection{The Kähler setting}

Let $(X,\omega)$ be an integral Kähler manifold and fix a smooth
volume form $dV$ on $X.$ For a given sequence $a_{k}\in\R$ we consider
the following Moser-Trudinger type functional on $\mathcal{H}(X,\omega):$
\[
\mathcal{G}_{a_{k}}(u):=\frac{1}{k}\log\int e^{-ku}dV+\frac{1}{V}\int u\frac{\omega^{n}}{n!}-\frac{k^{n}}{a_{k}}J_{\omega}(u)\]
which is $\R-$invariant (and hence descends to a functional on space
of all Kähler metrics in $[\omega]).$ We let $a_{k}(X)$ be the infimum
over all $a_{k}$ such that the functional above is bounded from above.
By Theorem \ref{thm:aubin hyp intro} (and the discussion in the beginning
of section \ref{sec:Remarks-on-the}) $a_{k}(X)\geq1/A$ or more precisely
$\liminf_{k}a_{k}(X)/k^{n+1}\geq1.$ 

In this section we will be concerned with the question of existence
of maximizers for $\mathcal{G}_{a_{k}}$ and solutions to the corresponding
Euler-Lagrange equation \begin{equation}
0=(d\mathcal{G}_{a_{k}})_{|u}=-\frac{e^{-ku}dV}{\int e^{-ku}dV}+\frac{1}{V}\frac{\omega^{n}}{n!}+\frac{k^{n}}{a_{k}}(\frac{\omega_{u}}{n!}-\frac{\omega^{n}}{n!})\label{eq:variational eq in kahler}\end{equation}
Braking the \textbackslash{}$\R-$invariance by the introducing the
normalization $\int_{X}e^{-ku}dV=V$ the previous equation can hence
be written as the following PDE: \begin{equation}
\frac{\omega_{u}^{n}}{n!}=\frac{a_{k}}{k^{n}}e^{-ku}dV+(1-\frac{a_{k}}{Vk^{n}})\frac{\omega^{n}}{n!}\label{eq:m-a eq in kahler setting}\end{equation}
for $u\in\mathcal{H}(X,\omega).$
\begin{thm}
If $a_{k}<a_{k}(X)$ and $a_{k}<Vk^{n}$ then there is a solution
to \ref{eq:m-a eq in kahler setting} in $\mathcal{H}(X,\omega).$
Moreover, the solution can be taken to maximize the functional $\mathcal{G}_{a_{k}}.$
In particular, if $a_{k}=a<1$ then there is such a solution for all
$k$ sufficiently large. 
\end{thm}
Given the Moser-Trudinger inequalities in Theorem \ref{thm:aubin hyp intro}
the proof of the previous theorem follows from the variational approach
to complex Monge-Ampère equation introduced in \cite{bbgz}. 

\emph{Existence of a maximizer $u_{*}$ in $\mathcal{E}^{1}(X,\omega)$ }

We proceed in two steps. The first step amounts to the following\emph{
coercivity estimate:} there exists $\delta,C>0$ such that \begin{equation}
\mathcal{G}_{a_{k}}\leq\delta\mathcal{E}_{\omega}+C\label{eq:coerc estimate kahler}\end{equation}
on the space $\mathcal{E}_{0}^{1}(X,\omega):=\mathcal{E}^{1}(X,\omega)\cap\{\sup_{X}=0\}$
(which we equip with the $L^{1}-$topology). This follows directly
from the assumption that $a<a(X)$ and the inequality \ref{eq:sup estimate for psh}.
The second step is to establish the following \emph{semi-continuity}
\emph{property: }for any constant $C$ the functional $\mathcal{G}_{a_{k}}$
is upper semi-continuous (usc) on $\{-\mathcal{E}_{\omega}\leq C\}$
in $\mathcal{E}_{0}^{1}(X,\omega)$ (wrt the $L^{1}-$topology). To
this end first recall that $\mathcal{E}_{\omega}$ is usc on $PSH(X,\omega)$
(in particular it follows from weak compactness that $\{-\mathcal{E}_{\omega}\leq C\}$
is compact) \cite{bbgz,begz}. All that remains is then to prove that
$u\mapsto\int e^{-ku}dV$ is usc on $\{-\mathcal{E}_{\omega}\leq C\}.$
To this end it is enough to established a uniform bound \begin{equation}
\int e^{-(k+\delta)u}\leq C_{\delta}\label{eq:upper bound on exponential}\end{equation}
 for some $\delta>0$ (compare the proof of Lemma 6.4 in \cite{bbgz}
or Lemma 3.6 in \cite{berm2}). But since we have assumed that $u$$\in\{-\mathcal{E}_{\omega}\leq C\}$
this is an immediate consequence of the Moser-Trudinger inequality
in Theorem \ref{thm:aubin hyp intro} (which shows that any $\delta>0$
will do). The existence of a maximizer $u_{*}$ is now rather immediate:
take $u_{j}$ in $\mathcal{E}_{0}^{1}(X,\omega)$ such that \[
\mathcal{G}_{a_{k}}(u_{j})\rightarrow\sup_{\mathcal{E}^{1}(X,\omega)}\mathcal{G}_{a_{k}},\]
(note that, by the scale invariance of $\mathcal{G}_{a_{k}}$ we may
indeed assume that $\sup_{X}u_{j}=0).$ By the coercivity estimate
the sup is finite and moreover $(u_{j})\subset\{-\mathcal{E}_{\omega}\leq C\}$
for some $C>0.$ But then it follows from the upper semi-continuity
that the sup is attained on any accumulation point $u_{*}$ of $(u_{j})$
(which exists by compactness). This concludes the proof of the existence
of a maximizer.

\emph{The maximizer $u_{*}$ is a weak solution of equation \ref{eq:m-a eq in kahler setting}.}

We will use the projection argument in \cite{begz} to see that $u_{*}$
is a (weak) solution in $\mathcal{E}_{0}^{1}(X,\omega)$ to the variational
equation \ref{eq:variational eq in kahler} (shifting $u_{*}$ by
a constant hence gives a solution to the equation \ref{eq:m-a eq in kahler setting}).
To this end we first decompose \[
\mathcal{G}_{a_{k}}(u)=\frac{k^{n+1}}{a_{k}}\mathcal{E}_{\omega}+\mathcal{I}_{a_{k}},\,\,\,\,\,(\mathcal{I}_{a_{k}}(u)=\log\int e^{-ku}dV+k(1-\frac{k^{n}}{Va_{k}})\int u\omega^{n}/n!)\]
Fixing $v\in\mathcal{C}^{\infty}(X)$ let $f(t):=\mathcal{E}_{\omega}(P_{\omega}(u_{*}+tv)+\mathcal{I}_{a_{k}}(u_{*}+tv),$
where 

\[
P_{\omega}(u)(x):=\sup\{v(x):\, v\leq u,\,\, v\in PSH(X,\omega)\}\]
By the assumption $a_{k}<k^{n}V$ the functional $\mathcal{I}_{a_{k}}(u)$
is decreasing in $u$ and hence the sup of $f(t)$ on $\R$ is attained
for $t=0.$ Now $\mathcal{E}_{\omega}\circ P_{\omega}$ is differentiable
with differential $MA(P_{\omega}u)$ at $u$ \cite{bbgz}. Hence,
the condition $df/dt=0$ for $t=0$ gives that the variational equation
\ref{eq:variational eq in kahler} holds when integrated against any
$v\in\mathcal{C}^{\infty}(X).$ 

\emph{Regularity}

Now, by the previous estimate \ref{eq:upper bound on exponential}
$\omega_{u_{*}}^{n}$ has a density in $L^{p}$ for some $p>1$ (or
even all $p>1)$ and hence it follows from Kolodziejs $L^{\infty}-$estimate
\cite{ko} that $u_{*}$ is in $L^{\infty}(X)$ (and is even continuous).
Finally the higher order regularity $u\in\mathcal{C}^{\infty}(X)$
then follows from \cite{sz-to}, using that the rhs in equation \ref{eq:m-a eq in kahler setting}
is of the form $F(u)$ for $F(t)$ smooth and positive (using the
assumption $a_{k}<k^{n}V).$

\subsection{Remarks on the Fano setting }

Let now $X$ be Fano with $[\omega]=c_{1}(-K_{X}).$ In the case when
$k=1$ and $a_{k}:=V$ the functional $\mathcal{G}_{a_{k}}$ above
becomes \[
\mathcal{G}_{a_{k}}:=\mathcal{G}_{V}(u):=\log\int e^{-u}dV+\frac{1}{V}\mathcal{E}_{\omega}(u)\]
with Euler-Lagrange equation \[
\omega_{u}^{n}/n!=Ve^{-u}dV\]
In particular, if $dV$ is taken as $e^{-h_{\omega}}\omega^{n}/n!$
where $h_{\omega}$ is the Ricci potential of $\omega$ then the previous
equation may be written as the Kähler-Einstein equation \[
(dd^{c}\phi)^{n}/n!=Ve^{-\phi}dz\wedge d\bar{z}\]
for the local weight $\phi$ of the metric $\omega,$ saying that
$\mbox{Ric \ensuremath{\omega=\omega}.}$ In this setting it is well-known
that the corresponding coercivity estimate \ref{eq:coerc estimate kahler}
is \emph{equivalent }to the existence of a Kähler-Einstein metric,
which in turn is equivalent to $X$ being {}``analytically $K-$stable''
in the sense of Tian (which means that Mabuchi's K-energy functional
is proper); see \cite{ti-1} Thm 7.13 and \cite{p-s+}.

Now, the coercivity estimate holds for $\mathcal{G}_{V}$ precisely
when a Moser-Trudinger inequality holds for some $a_{k}:=a$ (i.e.
$\mathcal{G}_{a}\leq C)$ satisfying \begin{equation}
V<a\label{eq:existence crituerium in terms of v}\end{equation}
In other words, if $a$ could be chosen uniformly over all Fano manifolds
$X$ of dimension $n$ then the previous inequality would give an
existence criterion for Kähler-Einstein metrics on .$X,$ in terms
of the volume of $X.$ This follows for example from the variational
approach above, but a proof using the continuity method already appears
in Aubin's paper \cite{au} (see also \cite{din} where the functional
$\mathcal{G}_{V}$ seems to first have appeared explicitly). As explained
in section \ref{sec:Remarks-on-the} Aubin also proposed an explicit
value for $a,$ which however cannot be correct. 

Unfortunately, it can be shown that the uniform constant provided
by Theorem \ref{thm:uniform fano intro} (at least in its present
form) is not useful for this kind of application. On the other hand
the existence of Moser-Trudinger type inequalities established in
Theorems \ref{thm:aubin hyp intro} and \ref{thm:uniform fano intro}
are very useful in other regards, for example for establishing semi-continuity
properties and uniform estimates as in the previous section. In particular,
it plays an important role in \cite{bbegz} in the construction of
Kähler-Einstein metrics on {}``analytically K-stable'' log-Fano
varieties.

Before turning to the setting of domains in $\C^{n}$ we briefly recall
Tian's \cite{ti1} existence criterion for Kähler-Einstein metrics
which has proved to be very useful: \begin{equation}
\alpha(X)>n/(n+1),\label{eq:tian's cond}\end{equation}
 where \[
\alpha(X):=\sup\left\{ t:\exists C_{t}:\,\int_{X}e^{-t(u-\sup_{X})}dV\leq C_{t},\,\forall u\in PSH(X,\omega)\right\} \]
 As is well-known it is enough to consider $u$ with analytic singularities
in the sup above (and hence $\alpha(X)$ coincides with the algebraically
defined log canonical threshold $\mbox{lc\ensuremath{(X))}.}$ Now,
if it would be enough to take the sup above over all $u$ with\emph{
isolated} singularities, then it would follow from the inequality
\ref{eq:local alg ineq} (see also below) that \[
\alpha(X)>n/(n!V)^{1/n}\]
 and hence Tian's criterion \ref{eq:tian's cond} would be satisfied
if $n!V<(n+1)^{n}.$ However, this latter condition is satisfied for
\emph{any} Fano manifold when $n=2$ (i.e. Del Pezzo surfaces) and
in particular for those which do not admit a Kähler-Einstein metric
(like $\P^{2}$ blown-up in one point) Still, as we will see next
a similar approach turns out to be very fruitful in the setting of
domains. At least on a heuristic level this could perhaps be expected
as all analytic singularities are indeed isolated in this setting.

\subsection{\label{sub:mf eq domain}The setting of domains in $\C^{n}$ and
Mean Field Equations}

Let now $\Omega$ be a hyperconvex domain in $\C^{n}$ with $dV$
the Euclidean volume form and recall (see section \ref{sub:Quasi-B-M-in})
that \[
\mathcal{G}_{\gamma}(u):=\frac{1}{\gamma}\log\int_{\Omega}e^{-\gamma u}dV+\frac{1}{n+1}\int(-u)(dd^{c}u)^{n}\]
so that the corresponding Euler-Lagrange equation reads \begin{equation}
(dd^{c}u)^{n}=\frac{e^{-\gamma u}dV}{\int_{\Omega}e^{-\gamma u}dV}\label{eq:m-a eq in domain}\end{equation}
with the boundary condition $u=0.$ Equivalently setting $v=\gamma u$
gives the Euler-Lagrange equation corresponding to the non-scaled
Moser-Trudinger inequality $(M-T)$ in the beginning of section \ref{sec:Moser-Trudinger-and-Demailly}
(it is obtained by setting $\gamma=1$ and inserting a multiplicative
constant $a=\gamma^{n}$ in the rhs). Ideally, we would like to look
for smooth solutions (in $\mathcal{H}_{0}(\Omega)$) to the previous
equation, but as the corresponding higher order regularity theory
does not seem to be sufficiently developed we will merely be able
to produce continuous solutions (vanishing on the boundary). Note
that in this setting there is no invariance under additive scalings
of $u$ (due to the boundary conditions $u=0$). 

In the case when $n=1$ the previous equation is often referred to
as the \emph{mean field equation} as it appears in a statistical model
of mean field type, with $\gamma$ playing the role of (minus) the
temperature \cite{clmp,k}. In the one-dimensional case it is well-known
that $\gamma=2$ appears as a critical value/phase transition (the
value is $8\pi$ when $dd^{c}$ is replaced by the usual non-normalized
Laplacian in the plane). It should be emphasized that the statistical
mechanical point of view only the solutions maximizing the corresponding
free energy functional are relevant. 
\begin{thm}
\label{thm:existence of solutions in domain}Let $\Omega$ be a hyperconvex
domain and assume that $\gamma<n+1.$ Then there exists $u_{\gamma}\in\mathcal{C}^{0}(\bar{\Omega})$
solving equation \ref{eq:m-a eq in domain} in $\Omega$ with $u_{\gamma}=0$
on $\partial\Omega$ and which maximizes the corresponding functional
$\mathcal{G}_{\gamma}.$\end{thm}
\begin{proof}
Assume that $\gamma<n+1.$ By Theorem \ref{thm:quasi-sharp m-t in ball intro}
the coercivity estimate corresponding to \ref{eq:coerc estimate kahler}
still holds for $\mathcal{G}_{\gamma}$ and it is well-known that
$\mathcal{E}$ is usc and its sub-level sets $\{\mathcal{E\geq}-C\}$
are compact (wrt the $L_{loc}^{1}-$topology); see \cite{a-c-c} and
references therein. Hence, all the previous arguments still apply
in the present setting of domains to give the existence of a maximizer
$u_{\gamma}$ for $\mathcal{G}_{\gamma}$ on the space $\mathcal{E}^{1}(\Omega).$
To see that $u_{\gamma}$ satisfies the equation \ref{eq:m-a eq in domain}
one applies a projection argument as in the Kähler setting above (see
\cite{a-c-c} where the projection argument from \cite{bbgz} was
adapted to the setting of hyperconvex domains). Finally, by the M-T
inequality $MA(u)$ has an $L^{p}-$density for $p>1$ and hence when
$\Omega$ is strictly pseudoconvex the continuity statement follows
from \cite{ko}, or alternatively from \cite{c-p} by taking $p=2$
(using the uniqueness of solution to the inhomogeneous Monge-Ampère
equation in the class $\mathcal{E}_{1}(\Omega)).$ As for the general
hyperconvex case it follows from \cite{bl}.
\end{proof}
Next we will establish a {}``concentration/compactness principle''
for the behavior of the solutions above when $\gamma$ approaches
the critical value $n+1.$ First recall that if $u$ psh in a neighborhood
of a point $z_{0}$ then its \emph{complex singularity exponent $c_{z_{0}}(u)$}
at $z_{0}$ is defined as \[
c_{z_{0}}(u):=\sup\left\{ t:\,\int_{U}e^{-tu}dV<\infty\right\} ,\]
 for $U$ some neighborhood of $z_{0}.$ As shown in \cite{ce} (Thm
5.5) the Brezis-Merle type inequality proved there may be localized
to give that \[
c_{z_{0}}(u)\geq n/(\int_{\{z_{0}\}}(dd^{c}u)^{n})^{1/n}\]
for any point $z_{0}\in\Omega$ and function $u\in\mathcal{F}(\Omega)$
(more generally the boundary assumptions on $u$ are not needed).
This can be seen as a generalization of the local algebra inequality
\ref{eq:local alg ineq} which corresponds to the case when $u=\log(\sum_{i=1}^{m}|f_{i}|^{2})$
for holomorphic functions $f_{i}$ (determining the ideal $\mathcal{I}:=(f_{1},...,f_{m})\subset\mathcal{O}_{z_{0}}(\C^{n})).$ 
\begin{thm}
\label{thm:conc-comp}Let $\gamma_{j}$ be a sequence increasing to
$n+1$ and $u_{j}:=u_{\gamma_{j}}$ a sequence of solutions of equation
\ref{eq:m-a eq in domain} as in the previous theorem converging to
$u\in\mathcal{F}(\Omega))$ in the $L_{loc}^{1}-$topology (which
is always possible to find after passing to a subsequence), then precisely
one of the following two alternatives hold:
\begin{enumerate}
\item $u_{j}$ converges uniformly to a solution $u$ of equation \ref{eq:m-a eq in domain}
for $\gamma=n+1,$ maximizing the functional $\mathcal{G}_{n+1}.$
\item For any $\delta>0$ the sequence $\int_{\Omega}e^{-(n+\delta)u_{j}}dV$
is unbounded. 
\end{enumerate}
In the second case above either the sequence $u_{j}$ has a blow-up
point in $\partial\Omega,$ i.e there is a sequence of point $z_{j}$
in $\Omega$ converging to $z_{0}\in\partial\Omega$ such that $\lim_{j}u(z_{j})=-\infty$
or the limit $u$ satisfies \begin{equation}
(dd^{c}u)^{n}=\delta_{z_{0}}\label{eq:ma is dirac in thm conc-comp}\end{equation}
 for some point $z_{0}\in\Omega$ and moreover $c_{z_{0}}(u)=n$ \end{thm}
\begin{proof}
We will use the notation from section \ref{sub:Quasi-B-M-in}. First
note that since $\gamma\mapsto\mathcal{L}_{\gamma}$ is decreasing
we have that $\mathcal{G}_{\gamma}\leq\mathcal{G}_{\gamma_{*}}$ if
$\gamma<\gamma_{*}$ and in particular $\sup\mathcal{G}_{\gamma}\leq\sup\mathcal{G}_{\gamma_{*}}.$
Hence, for $\gamma_{1}<\gamma_{i}<n+1$ we get \begin{equation}
-C:=\mathcal{G}_{\gamma_{1}}(u_{1})\leq\mathcal{G}_{\gamma_{_{i}}}(u_{i})\leq\mathcal{G}_{n+1}(u_{i})\label{eq:lower bd on free energy}\end{equation}
Now if the second alternative in the theorem does not hold for $u_{j}$
then Lemma \ref{lem:bounded using alpha} shows that there exists
$\delta$ such that the free energy functional $F_{n+1+\delta}$ is
uniformly bounded from above along $(u_{j})$ and hence so are the
functionals $\mathcal{G}_{n+1+\delta}$ (as explained in connection
to Lemma \ref{lem:bounded using alpha}). Combined with the lower
bound \ref{eq:lower bd on free energy} this means that \[
\mathcal{E}(u_{j})\geq-C.\]
 Hence, the Moser-Trudinger inequality applied to a fixed $\gamma_{1}<n+1$
(i.e. the bound $\mathcal{G}_{\gamma_{1}}\leq C)$ shows that $\int e^{-pu_{j}}\leq C_{p}$
for any $p>0.$ But then it follows from general principles (for the
same reasons as in the Kähler case) that $\int e^{-pu_{j}}\rightarrow\int e^{-pu_{j}},$
i.e. $\left\Vert e^{-u}\right\Vert _{L^{p}(\Omega)}\rightarrow\left\Vert e^{-u}\right\Vert _{L^{p}(\Omega)}$
and even more precisely that \begin{equation}
e^{-u_{j}}\rightarrow e^{-u},\,\,\,\,\mbox{in\,}L^{p}(\Omega).\label{eq:conv of densities in l2}\end{equation}
In particular $\mathcal{L}_{n+1}(u_{j})\rightarrow\mathcal{L}_{n+1}(u),$
as $j\rightarrow\infty.$ Moreover, a similar argument shows that
$u$ is a maximizer of $\mathcal{G}_{u}$ and hence the projection
argument gives, as above, that $u$ solves the equation \ref{eq:m-a eq in domain}.
Moreover, the convergence \ref{eq:conv of densities in l2} for $p=2$
gives that the $L^{2}(\Omega)-$norm of the densities $MA(u_{j})/dV-MA(u)/dV$
tend to zero and hence the stability result in \cite{c-p} show that
$u_{j}\rightarrow u$ in $L^{\infty}(\Omega).$

Finally, if there is no blow-up point in $\partial\Omega,$ then there
is a constant $M$ and a compact subset $K$ of $\Omega$ such that
$u\geq-M$ on $\Omega-K$ and hence $\int_{K}e^{-(n+\delta)u_{\gamma}}dV$
is unbounded. Now, if $u_{\gamma}\rightarrow u$ in $L_{loc}^{1},$
then it follows that $\int_{\Omega}(dd^{c}u)^{n}\leq1$ (see for example
the appendix in \cite{de}). By the semi-continuity of complex singularity
exponents \cite{d-k} there is a neighborhood $U$ of $K$ such that
$\int_{U}e^{-(n+\delta)u}=\infty$ for any $\delta>0,$ i.e. $c_{z}(u)\leq n,$
for any $z\in K.$ But then it follows from \ref{eq:local alg ineq}
that for any $z\in K$ \[
\int_{\{z\}}(dd^{c}\phi)\geq1.\]
Since $\int_{\Omega}(dd^{c}u)^{n}\leq1$ this forces the equation
\ref{eq:ma is dirac in thm conc-comp} to hold for some $z=z_{0}.$
Moreover, since $\int_{\Omega}(dd^{c}u)^{n}\leq1$ we already know,
by the quasi-sharp B-M inequality that, for any $\delta>0$ $e^{-(n-\delta)\phi}$
is in $L^{1}(\Omega)$ and hence $c_{z_{0}}(u)\geq n$. All in all
this means that $c_{z_{0}}(u)=n$ and that ends the proof.\end{proof}
\begin{rem}
In the case when $n=1$ it is well-known that there cannot be any
blow-points on $\partial\Omega$ (see Prop 4 in \cite{m-w}) and we
expect this to be true in general. It also seems natural to conjecture
that the limit $u$ in the second alternative above coincides with
the pluricomplex Green function $g_{z_{0}}$ with a pole at $z_{0}.$
This would in fact follow if $u$ were known a priori to have analytic
singularities at $z_{0},$ i.e $u(z)=\lambda\log(\sum_{i=1}^{m}|f_{i}|^{2})+O(1)$
close to $z_{0},$ where $\lambda\in\R$ and $f_{i}$ are holomorphic.
Indeed, since $c_{z_{0}}(u)=n/\int_{\{z_{0}\}}(dd^{c}u)^{n}$ it would
then follow from the equality case in the inequality \ref{eq:local alg ineq}
(see \cite{de fer-}) that $u(z)=\log|z-z_{0}|^{2}+O(1)$ close to
$z_{0}$ and hence $u=g_{z_{0}}$ by the comparison principle (at
least if a priori $u(z)\rightarrow0$ as $z\rightarrow\partial\Omega).$ 
\end{rem}
When $n=1$ it is well-known that the question whether there exists
solutions of equation \ref{eq:m-a eq in domain} in the critical case
$\gamma=2$ depends on the geometry of $\Omega$ (see \cite{clmp}).
For example, for the disc there is no solution, while there is one
for an annulus. In the case of a general $n$ the sharpness part of
Theorem \ref{thm:quasi-sharp m-t in ball intro} gives that, in the
super critical case $\gamma>n+1,$ the functional $\mathcal{G}_{a}$
is not bounded from above and in particular it has no maximizers (i.e.
the last part of Theorem \ref{thm:existence of solutions in domain}
cannot hold in this range). As for the critical case $\gamma=n+1$
one would expect that there is no solution of the equation \ref{eq:m-a eq in domain}
when $\Omega$ is the ball. For radial solutions this is straight-forward
to check. Indeed, an explicit calculation then reveals that, for any
$\gamma<n+1,$ a radial solution $u_{\gamma}$ is uniquely determined
and hence given by $u_{\gamma}=\phi_{0}^{\epsilon}$ (formula \ref{eq:rescaled f-s in unit-ball})
for some $\epsilon,$ where $\gamma\rightarrow n+1$ corresponds to
$\epsilon\rightarrow0.$ More over, when $\gamma=n+1$ there is no
radial solution and $u_{\gamma}\rightarrow(n+1)\log|z|^{2}$ as $\gamma\rightarrow n+1$
where $u$ has infinite energy, i.e. it is not an element in $\mathcal{E}^{1}(\Omega).$
In fact, in the case $n=1$ \emph{any} solution is radial, as follows
from the method of moving planes \cite{gnn} (which also applies to
the corresponding equation associated to the \emph{real} Monge-Ampère
operator \cite{del}). It hence seems natural to make the following
\begin{conjecture}
In the case of the ball in $\C^{n}$ any solution to equation \ref{eq:m-a eq in domain}
is radial and hence given by $u_{\gamma}$ above.
\end{conjecture}
If true the previous conjecture implies the validity of the sharp
Moser-Trudinger inequality (without assuming $S^{1}-$invariance),
i.e. that $\mathcal{G}_{\gamma}$ is bounded in the critical case
$\gamma=n+1.$ Indeed, given $u\in\mathcal{H}_{0}(\mathcal{B})$ we
have \[
\mathcal{G}_{\gamma}(u)=\lim_{\epsilon\rightarrow0}\mathcal{G}_{\gamma(\epsilon)}(u)\leq\lim_{\epsilon\rightarrow0}\sup\mathcal{G}_{\gamma(\epsilon)}\]
But by the previous theorem the sup of $\mathcal{G}_{\gamma(\epsilon)}$
is attained for some function $u_{\gamma(\epsilon)}$ satisfying the
equation \ref{eq:m-a eq in domain}, which if the conjecture above
is correct has to be radial and thus coincides with $\phi_{0}^{\epsilon}$
above. Finally, as shown towards the end in section \ref{sec:M-T ineq in ball under invariance-1}
$\mathcal{G}_{a}(\phi_{0}^{\epsilon})\rightarrow C_{n}$ and hence
$\mathcal{G}_{a}(u)\leq C_{n}.$ Note also that by Theorem \ref{thm:sharp m-t in ball under inv intro}
it would be enough to know that any solution is $S^{1}-$invariant
in order to deduce the sharp Moser-inequality using the previous argument.


\begin{thebibliography}{55}
\bibitem{ce}Ahag, P.; Cegrell, U.; Ko\l{}odziej, S.; Ph\d{a}m, H.
H.; Zeriahi, A. Partial pluricomplex energy and integrability exponents
of plurisubharmonic functions. Adv. Math. 222 (2009), no. 6, 2036\textendash{}2058. 

\bibitem{a-c-c}Ahag, P; Cegrell, U; Czyz, R: On Dirichlet's principle
and problem. arXiv:0912.1244 

\bibitem{au}Aubin, T: Réduction du cas positif de l´equation de Monge-Amp`ere
sur les variétés kahleriennes compactes à la démonstration d'une inégalité,
Journal of Functional Analysis 57 (1984), 143-153.

\bibitem{au2}Aubin, T: Some nonlinear problems in Riemannian geometry.
Springer-Verlag, Berlin, 1998.

\bibitem{Bec}Beckner, W: Sharp Sobolev inequalities on the sphere
and the Moser-Trudinger inequality. Annals of Math. 138 (1993), 213-242.

\bibitem{b-g-z}Benelkourchi, S; Guedj, V; Zeriahi, A: Plurisubharmonic
functions with weak singularities. Complex analysis and digital geometry,
57\textendash{}74, Acta Univ. Upsaliensis Skr. Uppsala Univ. C Organ.
Hist., 86, Uppsala Universitet, Uppsala, 2009

\bibitem{b-d}Berman, R.J:; Demailly, J-P: Regularity of plurisubharmonic
upper envelopes in big cohomology classes. Arkiv för Matematik (to
appear) arXiv:0905.1246 

\bibitem{berm1}Berman, R.J: Analytic torsion, vortices and positive
Ricci curvature. arXiv:1006.2988 

\bibitem{berm2}Berman, R.J: A thermodynamical formalism for Monge-Ampere
equations, Moser-Trudinger inequalities and Kahler-Einstein metrics.
arXiv:1011.3976 

\bibitem{bbgz}Berman, R.J.: Boucksom, S; Guedj, V; Zeriahi, A: A
variational approach to complex Monge-Ampère equations. arXiv:0907.4490

\bibitem{bbegz}Berman, R.J.: Boucksom, S; Eyssidieux, P; Guedj, V;
Zeriahi, A: Ricci iterationa and Kähler-Ricci flow on log-Fano varities.
In preperation.

\bibitem{bern1}Berndtsson, B: Curvature of vector bundles associated
to holomorphic fibrations. Annals of Math. Vol. 169 (2009), 531-560 

\bibitem{bern1b}Berndtsson, B: Positivity of direct image bundles
and convexity on the space of Kähler metrics. J. Differential Geom.
Volume 81, Number 3 (2009), 457-482. 

\bibitem{bern2}Berndtsson, B: A Brunn-Minkowski type inequality for
Fano manifolds and the Bando-Mabuchi uniqueness theorem. arXiv:1103.0923 

\bibitem{bl}{]}Blocki, B: On the L p-stability for the complex Monge-Ampere
operator, Michigan Math. J, 1995

\bibitem{begz}Boucksom, S; Essidieux,P: Guedj,V; Zeriahi: Monge-Ampere
equations in big cohomology classes. Acta. Math. (to appear). arXiv:0812.3674

\bibitem{br-m}Brézis, H; Merle, F: Uniform estimates and blow-up
behavior for solutions of $-\Delta u=V(x)e^{u}$ in two dimensions,,
Communs Partial Diff. Eqns 16 (1992) (8 \& 9), pp. 1223\textendash{}1253. 

\bibitem{clmp}Caglioti.E; Lions, P-L; Marchioro.C; Pulvirenti.M:
A special class of stationary flows for two-dimensional Euler equations:
a statistical mechanics description. Communications in Mathematical
Physics (1992) Volume 143, Number 3, 501-525

\bibitem{ca}Campana, F: Connexit´e rationnelle des vari´et´es de
Fano. Ann. Sci. ´Ecole Norm. Sup. (4) 25 (1992), no. 5, 539\textendash{}545

\bibitem{ce0}Cegrell, U: Pluricomplex energy. Acta Math. 180 (1998),
no. 2, 187\textendash{}217

\bibitem{ce3}Cegrell, U: Measures of finite pluricomplex energy.
arXiv:1107.1899

\bibitem{c-z}Cegrell, U; Zeriahi A., Subextension of plurisubharmonic
functions with bounded Monge-Amp`ere mass. C. R. Acad. Sci. Paris,
Ser. I 336 (2003).

\bibitem{c-p}Cegrell, U; Persson, L: The Dirichlet problem for the
complex Monge-Ampère operator: stability inL 2.Michigan Math. J.,
39 (1992), 145\textendash{}151.

\bibitem{ce2}Cegrell, U: Approximation of plurisubharmonic functions
in hyperconvex domains. Complex analysis and digital geometry, 125\textendash{}129,
Acta Univ. Upsaliensis Skr. Uppsala Univ. C Organ. Hist., 86, Uppsala
Universitet, Uppsala, 2009.

\bibitem{de fer-}de Fernex, T., Ein, L; Mustat\v{ }a: Multiplicities
and log canonical thresholds. J. Algebraic Geom. 13 (2004) 603\textendash{}615.

\bibitem{del}Delanoe, P: Radially symmetric boundary value problems
for real and complex elliptic Monge-Ampère equations\textquotedbl{},
J. Diff. Eq. 58 (1985) 318-344.

\bibitem{de}Demailly, J-P: Estimates on Monge-Ampère operators derived
from a local algebra inequality. Complex analysis and digital geometry,
131\textendash{}143, Acta Univ. Upsaliensis Skr. Uppsala Univ. C Organ.
Hist., 86, Uppsala Universitet, Uppsala, 2009.

\bibitem{d-k}Demailly, J-P; Kollar, J: Semi-continuity of complex
singularity exponents and Kähler-Einstein metrics on Fano orbifolds.
Ann. Sci. École Norm. Sup. (4) 34 (2001), no. 4, 525\textendash{}556.

\bibitem{din}Ding, W.: Remarks on the existence problem of positive
Kähler-Einstein metrics. Math. Ann. 463--472 (1988).

\bibitem{d-t}Ding, W. and Tian, G.: The generalized Moser-Trudinger
Inequality. Proceedings of Nankai International Conference on Nonlinear
Analysis, 1993.

\bibitem{font}Fontana, L.: Sharp borderline Sobolev inequalities
on compact Riemannian manifolds. Comment. Math. Helv. 68 (1993), 415-454

\bibitem{f-m}Fontana, L; Morpurgo, C: Adams inequalities on measure
spaces. arXiv:0906.5103

\bibitem{gnn}Gidas, B;Ni, W. M.; Nirenberg, L: Symmetry and related
properties via the Maxi- mum principle, Comm. Math. Phys. 68 (1979),
209-243.

\bibitem{g-z}Guedj,V; Zeriahi, A: Intrinsic capacities on compact
Kähler manifolds. J. Geom. Anal. 15 (2005), no. 4, 607--639.

\bibitem{g-z2}Guedj,V; Zeriahi, A: The weigthed Monge-Ampère energy
of quasiplurisubharmonic functions. arXiv:math/0612630 

\bibitem{h}Hwang, M: On the degrees of Fano four-folds of Picard
number 1, J. Reine Angew. Math. 556 (2003), pp. 225\textendash{}235. 

\bibitem{il}Ilias, S: Constantes explicites pour les inégalités de
Sobolev sur les variétés riemanniennes compactes. Annales de l'institut
Fourier, 33 no. 2 (1983), p. 151-165 

\bibitem{k}Kiessling M.K.H.: Statistical mechanics of classical particles
with logarithmic interactions, Comm. Pure Appl. Math. 46 (1993), 27-56.

\bibitem{kl}Klimek, M: Pluripotential theory, Clarendon Press/Oxford
Univ. Press ( 1991

\bibitem{k-ma}Kollár, J; Matsusaka, T: Riemann-Roch Type Inequalities.
American J. of Math. Vol. 105, No. 1, Feb., 1983

\bibitem{kmm}Kollar, J; Miyaoka, Y; Mori, S: Rational connectedness
and boundedness of Fano manifolds. J. Diff{}erential Geom. 36 (1992),
no. 3, 765\textendash{}779.

\bibitem{ko}Ko\l{}odziej, S.: The complex Monge\textendash{}Ampère
equation. Acta Math. 180, 69\textendash{}117 (1998).

\bibitem{li}Li, P: On the Sobolev constant and the p-spectrum of
a compact. Riemannian manifold. Ann. scient. de l'É.N.S. 4e série,
tome 13, no 4 ( 1980), p. 451-468

\bibitem{m-w}Ma, L, Wei, J.C: Convergence for a Liouville equation.
Comment. Math. Helvetici (2001), 76 (3), pg. 506-514 

\bibitem{m}Moser, J. A sharp form of an inequality by N. Trudinger.
Indiana Univ. Math. J. 20 (1970/71), 1077\textendash{}1092. 

\bibitem{p-s+}Phong, D.H: Song, J; Sturm, J; Weinkove, B: The Moser-Trudinger
inequality on Kahler-Einstein manifolds. Amer. J. Math. 130 (2008),
no. 4, 1067-1085,  arXiv:math/0604076 

\bibitem{sz-to}Székelyhidi, G; Tosatti.V: Regularity of weak solutions
of a complex Monge-Ampère equation. To appear in Analysis \& PDE

\bibitem{ti1}Tian, G: On Kähler-Einstein metrics on certain K¨ahler
manifolds with C1(M) > 0, Inventiones Mathematicae 89 (1987), 225\textendash{}246.

\bibitem{ti2}Tian, G; On Calabi's conjecture for complex surfaces
with positive first Chern class. Invent. Math. Vol. 101, Nr. 1 (1990) 

\bibitem{ti-1}Tian, G: Canonical Metrics in Kähler Geometry, Birkh¨auser,
2000.

\bibitem{t-y}Tian, G; Yau, S-T: Kahler-Einstein metrics on complex
surfaces with C 1 (M) positive. Commun. Math. Phys. 112, 175-203 (1987)

\bibitem{tr}Trudinger, N.S: On imbeddings into Orlicz spaces and
some applications. J. Math. Mech. 17 1967 473\textendash{}483. 

\bibitem{ts}Tso, K: On symmetrization and Hessian equations. Journal
d'Analyse Mathématique, Vol. 52, Number 1, 94-106, 

\bibitem{z}Zelditch, S: Book review of {}``Holomorphic Morse inequalities
and Bergman kernels Journal'' (by Xiaonan Ma and George Marinescu)
in Bull. Amer. Math. Soc. 46 (2009), 349-361. 

\bibitem{zer}Zeriahi, A: Volume and capacity of sublevel sets of
a Lelong class of psh functions. Indiana Univ. Math. J. 50 (2001),
no. 1, 671\textendash{}703.
\end{thebibliography}
\end{document}